\documentclass[11pt]{amsart}
\usepackage{amssymb,amsthm,amsmath}
\usepackage{latexsym,amsfonts, graphicx}%,url}%ulem}
\usepackage[boxed,norelsize]{algorithm2e}
\usepackage{ytableau}
\usepackage[all]{xy}
\usepackage{multirow}
\usepackage{fullpage}
\theoremstyle{plain}% default 
\newtheorem{theorem}{Theorem}[section] 
\newtheorem{lemma}[theorem]{Lemma} 
\newtheorem{definition}{Definition}[section]
\newtheorem{example}{Example}[section]

\newtheorem{proposition}[theorem]{Proposition} 
\newtheorem{corollary}{Corollary}[section]
\newtheorem{remark}[theorem]{Remark}

\DeclareMathOperator{\block}{block}

\title{Cyclic Symmetry of the Scaled Simplex}
\author[H.\ Thomas]{Hugh Thomas}
\email{hthomas@unb.ca}

\author[N.\ Williams]{Nathan Williams}
\email{will3089@math.umn.edu}

\begin{document}

\begin{abstract}
Let $\mathcal{Z}_m^k$ consist of the $m^k$ alcoves contained in the $m$-fold dilation of the fundamental alcove of the type $A_k$ affine hyperplane arrangement.  As the fundamental alcove has a cyclic symmetry of order $(k+1)$, so does $\mathcal{Z}_m^k$.  By bijectively exchanging the natural poset structure of $\mathcal{Z}_{m}^k$ for a natural cyclic action on a set of words, we prove that $(\mathcal{Z}_{m}^k,\prod_{i=1}^{k} \frac{1-q^{m i}}{1-q^i},C_{k+1})$ exhibits the cyclic sieving phenomenon.
\end{abstract}

\maketitle

%------------------------------------------------------------------------------------------------------------------------
\section{Introduction}
\label{sec:intro}
%------------------------------------------------------------------------------------------------------------------------

Let $\mathcal{Z}_m^k$ consist of the $m^k$ alcoves contained in the $m$-fold dilation of the fundamental alcove of the type $A_k$ affine hyperplane arrangement.  As the fundamental alcove has a cyclic symmetry of order $(k+1)$, so does $\mathcal{Z}_m^k$.  Let $\mathcal{W}_m^k$ be the set of words of length $(k+1)$ on $\mathbb{Z}/m\mathbb{Z}$ with sum $(m-1)\pmod m$, with
the order $(k+1)$ cyclic action given by rotation.  As the orbit structure of $\mathcal{W}_{m}^k$ is easily understood, we determine the orbit structure of $\mathcal{Z}_m^k$ with the following theorem.

\begin{theorem}
\label{thm:forward}
	There is an equivariant bijection from $\mathcal{Z}_{m}^k$ under its cyclic action to $\mathcal{W}_{m}^k$ under rotation.
\end{theorem}

%[FINISH]: can we give this bijection directly?
%We may state the bijection immediately: given an alcove in $\mathcal{Z}_m^k$, let $H_{\alpha_i,p_i}$ for $0 \leq i \leq k$ be its bounding hyperplanes (we take $\alpha_0$ to be the highest root).

%\begin{enumerate}
%	\item If the alcove lies below $H_{\alpha_0,p_0}$, then it is labeled by $p_1 p_2 \ldots p_k (m-p_0)$.
%	\item If it instead lies above $H_{\alpha_0,p_0}$, then it is labeled by $p_0 (m-p_1)(m-p_2)\ldots (m-p_k)$.
%\end{enumerate}

Figure~\ref{fig:ex1} illustrates this bijection for $\mathcal{Z}_4^2$.

\begin{figure}[ht]
\begin{center}
\includegraphics[height=3.3in]{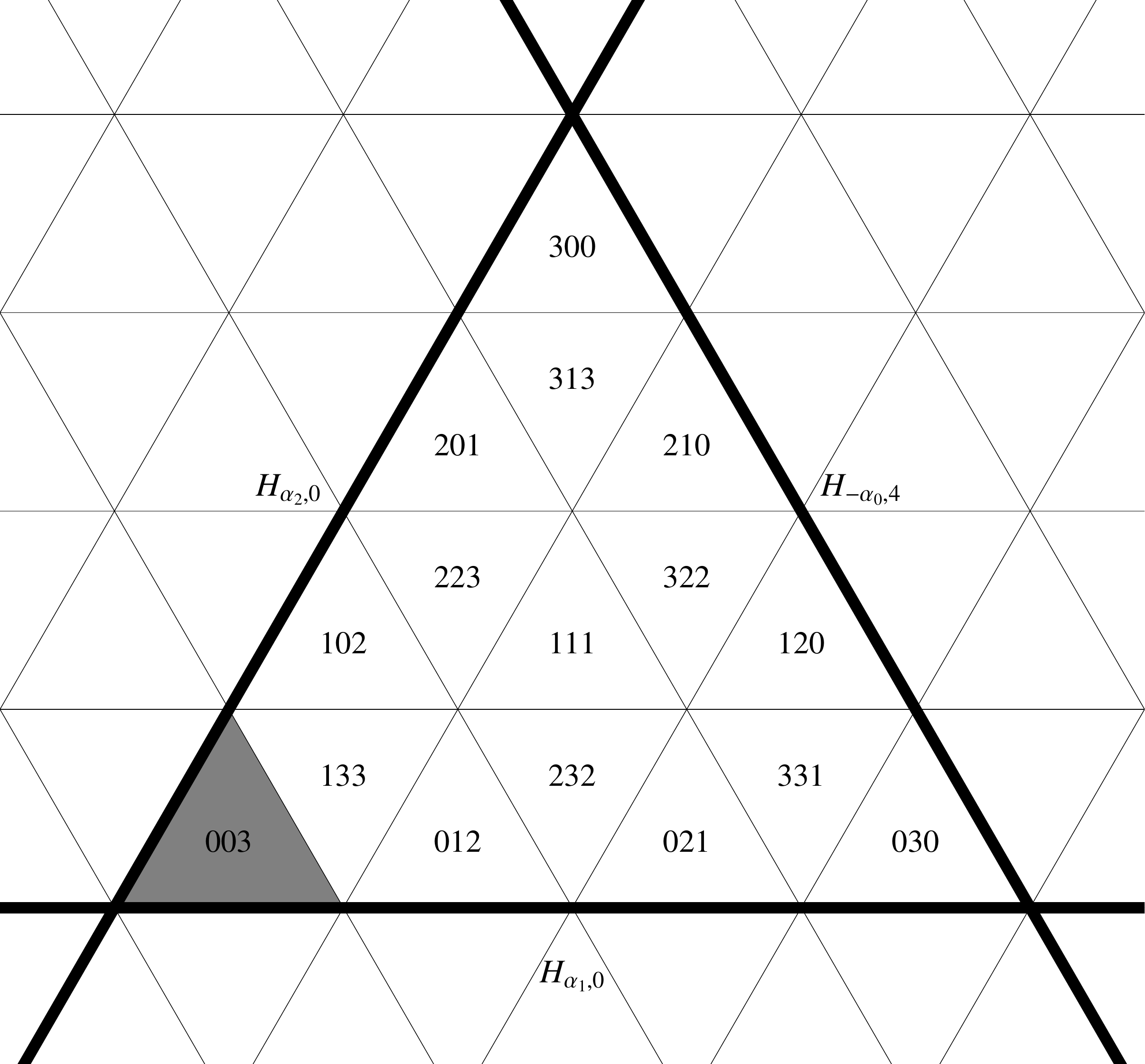}
\end{center}
\caption{The 16 alcoves of $\mathcal{Z}_4^2$ are those alcoves contained within the thick black lines.  Each alcove is labeled with its corresponding word in $\mathcal{W}_4^2$.  There are 5 orbits of size $3$ and a single orbit (the alcove in the center) of size $1$.}
\label{fig:ex1}
\end{figure}

The paper is organized as follows.  We give a brief history of the relevant work on this problem in Section~\ref{sec:history}.  This section also serves as an additional introduction to the paper.  In Sections~\ref{sec:alcores} and ~\ref{sec:posety}, we follow C.\ Berg and M.\ Zabrocki by interpreting $\mathcal{Z}_{m}^k$ as a poset $\mathcal{Y}_{m}^k$ on $(k+1)$ cores~\cite{berg2011symmetries}.  We do not give the cyclic action directly on the cores---in Section~\ref{sec:wordsthatend}, we give a combinatorial description of a cyclic action on a poset $\mathcal{X}_{m}^k$ on words of length $k$ on $\mathbb{Z}/m\mathbb{Z}$, and we then show that $\mathcal{X}_m^k$ is isomorphic to $\mathcal{Y}_m^k$.  In Section~\ref{sec:wordsthatsum}, we use the cyclic sieving phenomenon to analyze the orbit structure of $\mathcal{W}_m^k$ under rotation and we give the forward direction for an equivariant bijection between $\mathcal{X}_{m}^k$ and $\mathcal{W}_m^k$.  This bijection exchanges the natural poset structure on $\mathcal{X}_{m}^k$ for a natural cyclic action on $\mathcal{W}_m^k$.  To construct the more difficult inverse map, we first generalize in Section~\ref{sec:dendro} and then restrict in Section~\ref{sec:proof}.  %We conclude in Section~\ref{sec:conclusion} by giving the direct bijection from $\mathcal{Z}_m^k$ to $\mathcal{W}_m^k$ obtained by composing the above maps.

In summary, we prove Theorem~\ref{thm:forward} by showing that there are equivariant bijections between the following objects:

$$\mathcal{Z}_{m}^k \underset{\text{Sections~\ref{sec:alcores} and ~\ref{sec:posety}}}{\simeq} \mathcal{Y}_{m}^k \underset{\text{Section~\ref{sec:wordsthatend}}}{\simeq} \mathcal{X}_{m}^k \underset{\text{Sections~\ref{sec:wordsthatsum},~\ref{sec:dendro}, and~\ref{sec:proof}}}{\simeq} \mathcal{W}_{m}^k.$$

%$$\underset{\text{Alcoves under geometric rotation}}{\underbrace{\mathcal{Z}_{m}^k}} \simeq \underset{(k+1)-\text{cores}}{\underbrace{\mathcal{Y}_{m}^k}} \simeq \underset{\text{Words with a complicated rotation}}{\underbrace{\mathcal{X}_{m}^k}} \simeq \underset{\text{Words under rotation}}{\underbrace{\mathcal{W}_{m}^k}}.$$
\vspace{30pt}
%------------------------------------------------------------------------------------------------------------------------
\section{History}
\label{sec:history}
%------------------------------------------------------------------------------------------------------------------------

At a 2007 conference in Rome, R.\ Suter gave a talk in which he defined a surprising cyclic symmetry of order $(k+1)$ of a subposet $\mathcal{Y}^k_2$ of Young's lattice, for each $k \in \mathbb{N}$.  This was based on his work in 2002 to understand the abelian ideals of complex simple Lie algebras. The symmetry is due to the fact that the Dynkin diagram of $\tilde{A}_{k}$ is a cycle of length $(k+1)$ (see~\cite{suter2002young} and~\cite{suter2004abelian}).

To give the flavor of R.\ Suter's result, let $\mathcal{Y}^k_2$ be the subposet of Young's lattice containing those partitions $\lambda=(\lambda_1 \geq \lambda_2 \geq \ldots \geq \lambda_\ell)$ for which $\lambda_1+\ell \leq k$.  The symmetry is revealed by drawing the Hasse diagram of $\mathcal{Y}^k_2$ and then remembering only the underlying graph structure.  Figure~\ref{fig:suter} illustrates this for $k=4$.

\begin{figure}[ht]
\begin{center}
$\begin{array}{cc}
\begin{gathered}\includegraphics[height=2.5in]{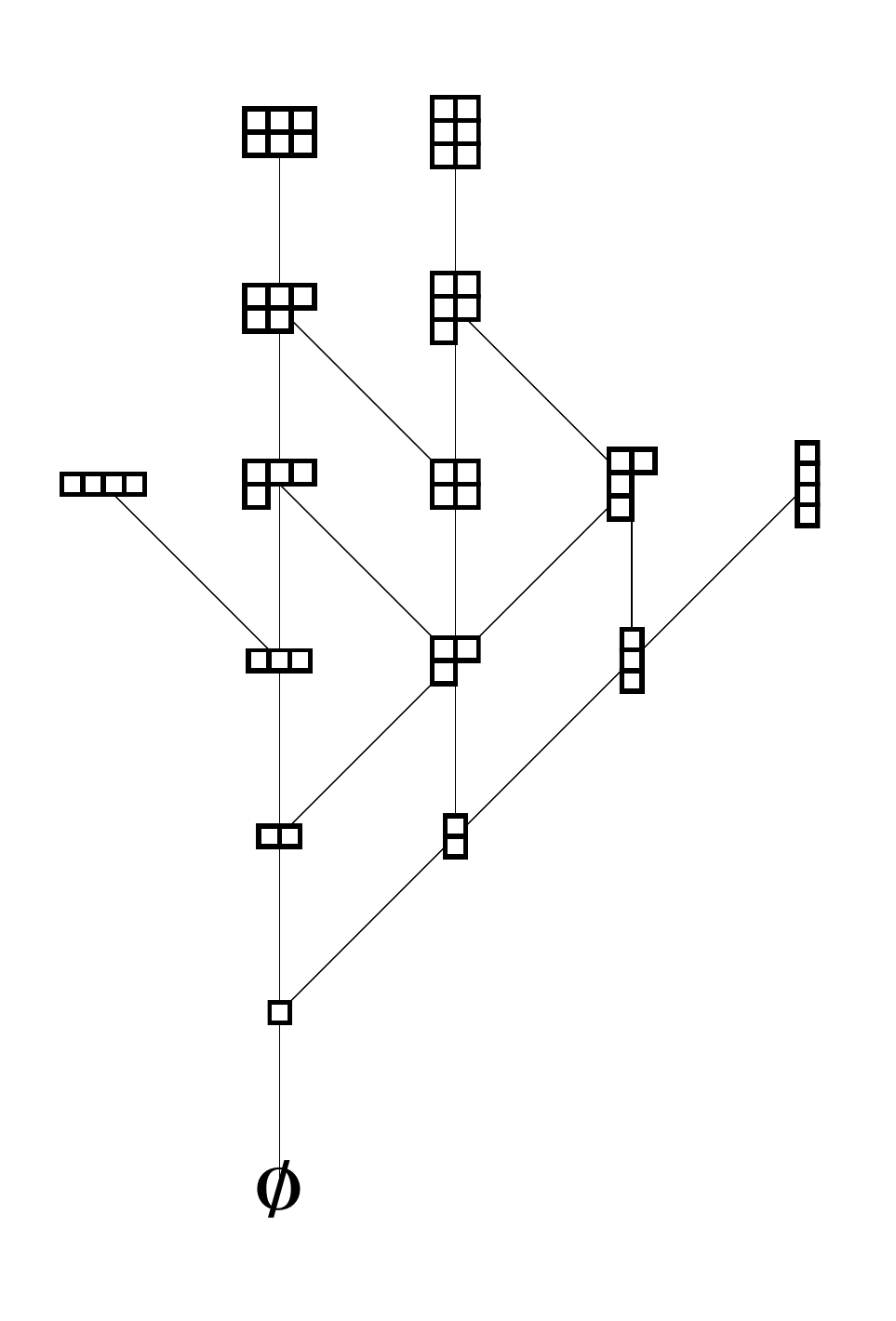}\end{gathered} &
\begin{gathered}\includegraphics[height=2.5in]{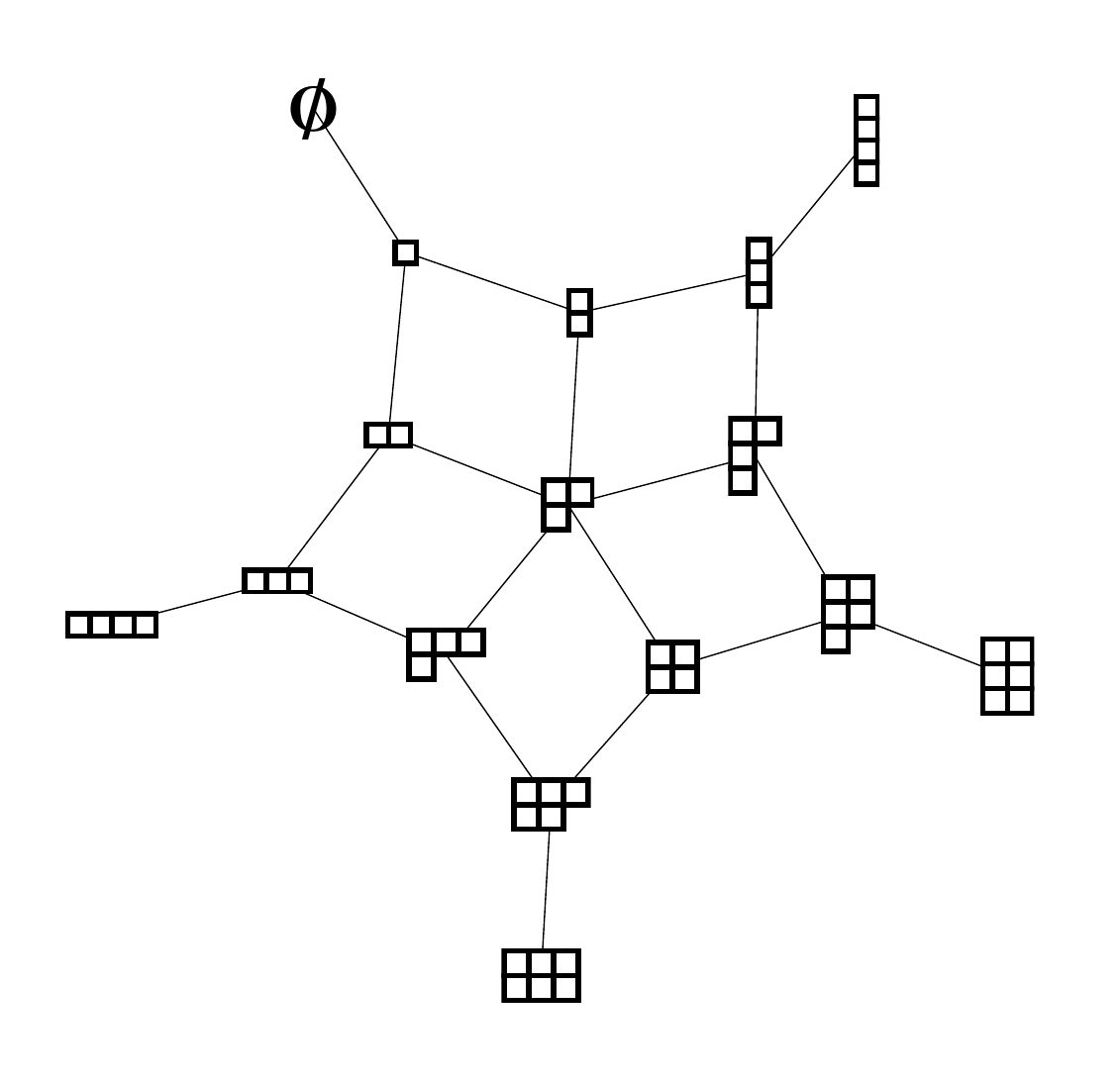}\end{gathered} \\
\end{array}$
\end{center}
\caption{The Hasse diagram of $\mathcal{Y}^4_2$ and its underlying graph.}
\label{fig:suter}
\end{figure}

After seeing the striking rotational symmetry on a transparency in R.\ Suter's talk, V.\ Reiner conjectured that the graphs would exhibit the cyclic sieving phenomenon~\cite{reinerpc}.  D.\ Stanton later refined this to the conjecture that there is an equivariant bijection between R.\ Suter's partitions under their cyclic action and binary words of length $(k+1)$ with odd sum under rotation~\cite{stantonpc}.

In 2010, the problem was presented in this form to the second author.  The result was proved by giving such a bijection in~\cite{williamsbijactions}.  This proof allowed for a natural \emph{combinatorial} generalization of the $(k+1)$ cyclic action to certain posets $\mathcal{X}_m^k$ on words of length $k$ on $\mathbb{Z}/m\mathbb{Z}$.  Figure~\ref{fig:x24} illustrates the cyclic symmetry of order 3 for $m=4$ and $k=2$ (compare to Figure~\ref{fig:ex1}).

\begin{figure}[ht]
\begin{center}
$\begin{array}{cc}
\begin{gathered}\includegraphics[height=2.4in]{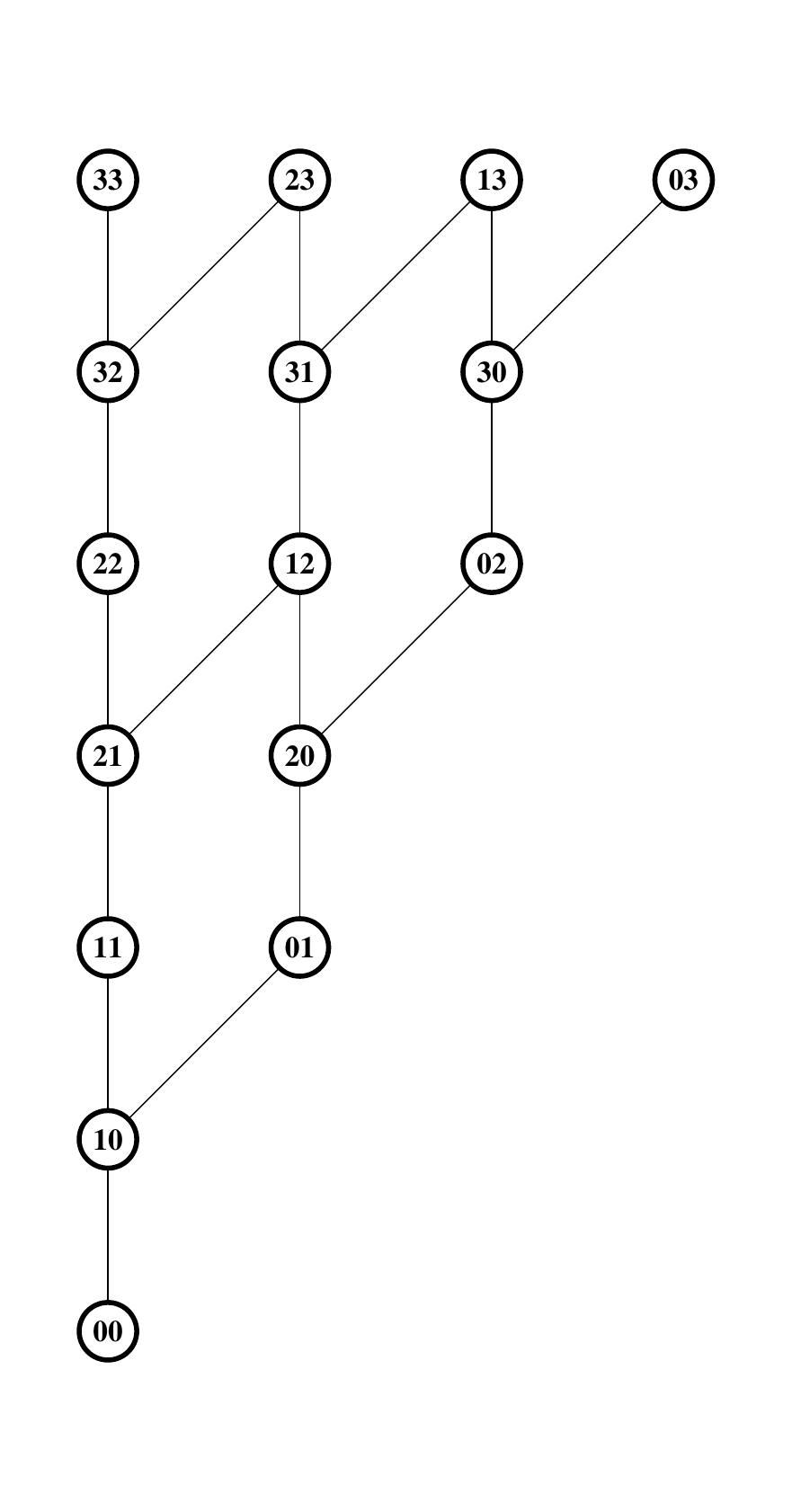}\end{gathered} &
\begin{gathered}\includegraphics[height=2.4in]{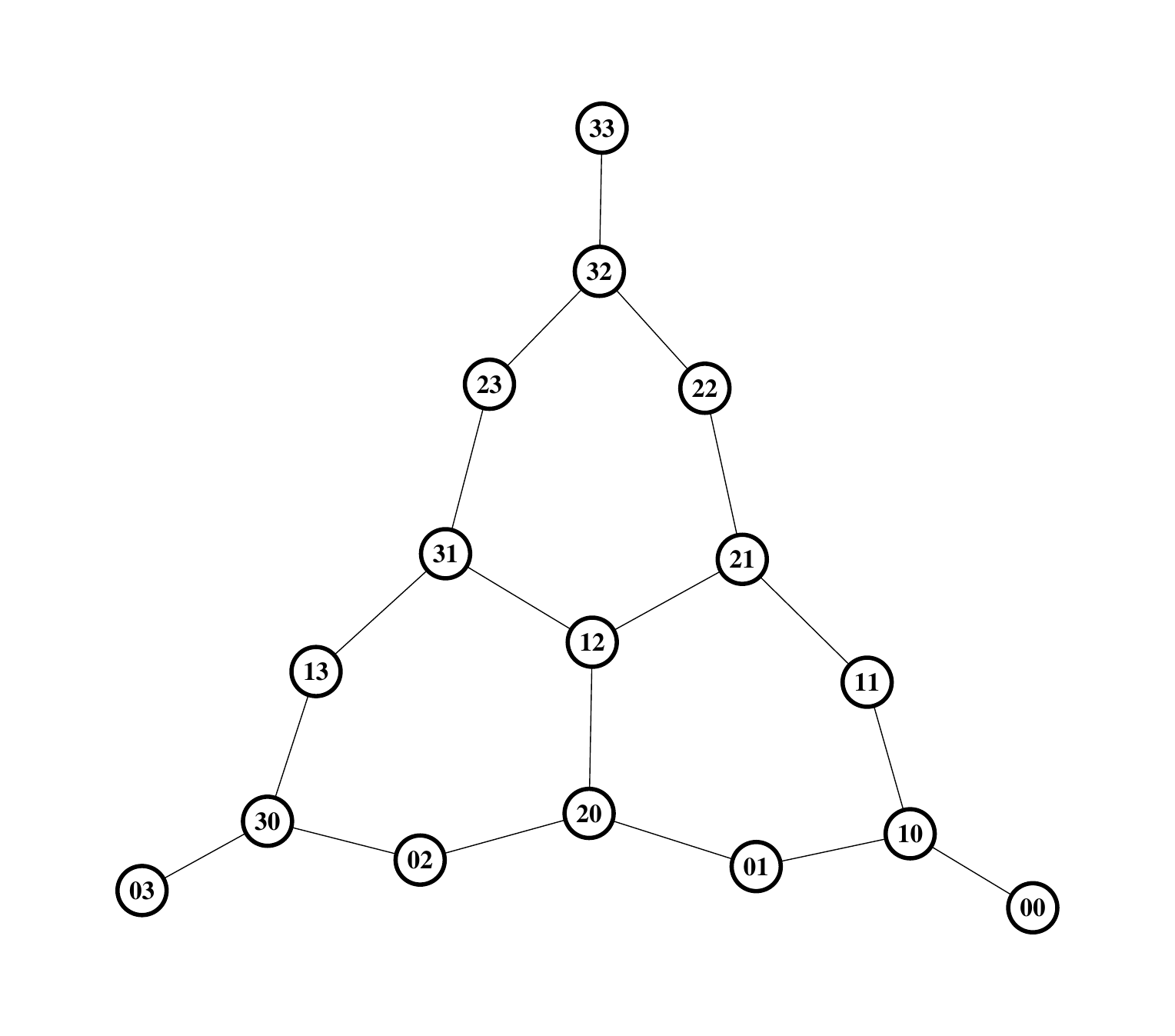}\end{gathered} \\
\end{array}$
\end{center}
\caption{The Hasse diagram of $\mathcal{X}^2_4$ and its underlying graph.}
\label{fig:x24}
\end{figure}

%$\mathcal{W}_m^k$ of length $(k+1)$ on $\mathbb{Z}/m\mathbb{Z}$ with sum $(m-1)\pmod m$ under rotation

No interpretation of the general posets $\mathcal{X}_m^k$ in terms of partitions was given in~\cite{williamsbijactions}.   Even more distressing was that the forward direction of the bijection was generalized to a map from $\mathcal{X}_m^k$ to $\mathcal{W}_m^k$, but the inverse was not found.  Shortly after this result, M.\ Visontai was able to solve a system of linear equations to give a proof that the map was invertible for $m=3$~\cite{visontaibijection}.

From the other direction, M.\ Zabrocki---while looking at the Wikipedia page for Young's Lattice---also came across R.\ Suter's result.  He and C.\ Berg gave a natural \emph{geometric} generalization of R.\ Suter's poset to $\mathcal{Z}_m^k$ as the $m$-fold dilation of the fundamental alcove in type $A_k$, from which the cyclic symmetry is intrinsically obvious~\cite{berg2011symmetries}.  They further gave the correct definition for the poset in terms of partitions, by letting $\mathcal{Y}^k_m$ be a certain order ideal in the $k$-Young's lattice of $(k+1)$-cores.  Figure~\ref{fig:y24} illustrates the cyclic symmetry of order 3 for $\mathcal{Y}^2_4$.

%\footnote{We thank A. Miller for pointing out that it was M. Hardy who made the relevant revisions to Wikipedia.}

\begin{figure}[ht]
\begin{center}
$\begin{array}{cc}
\begin{gathered}\includegraphics[height=2.4in]{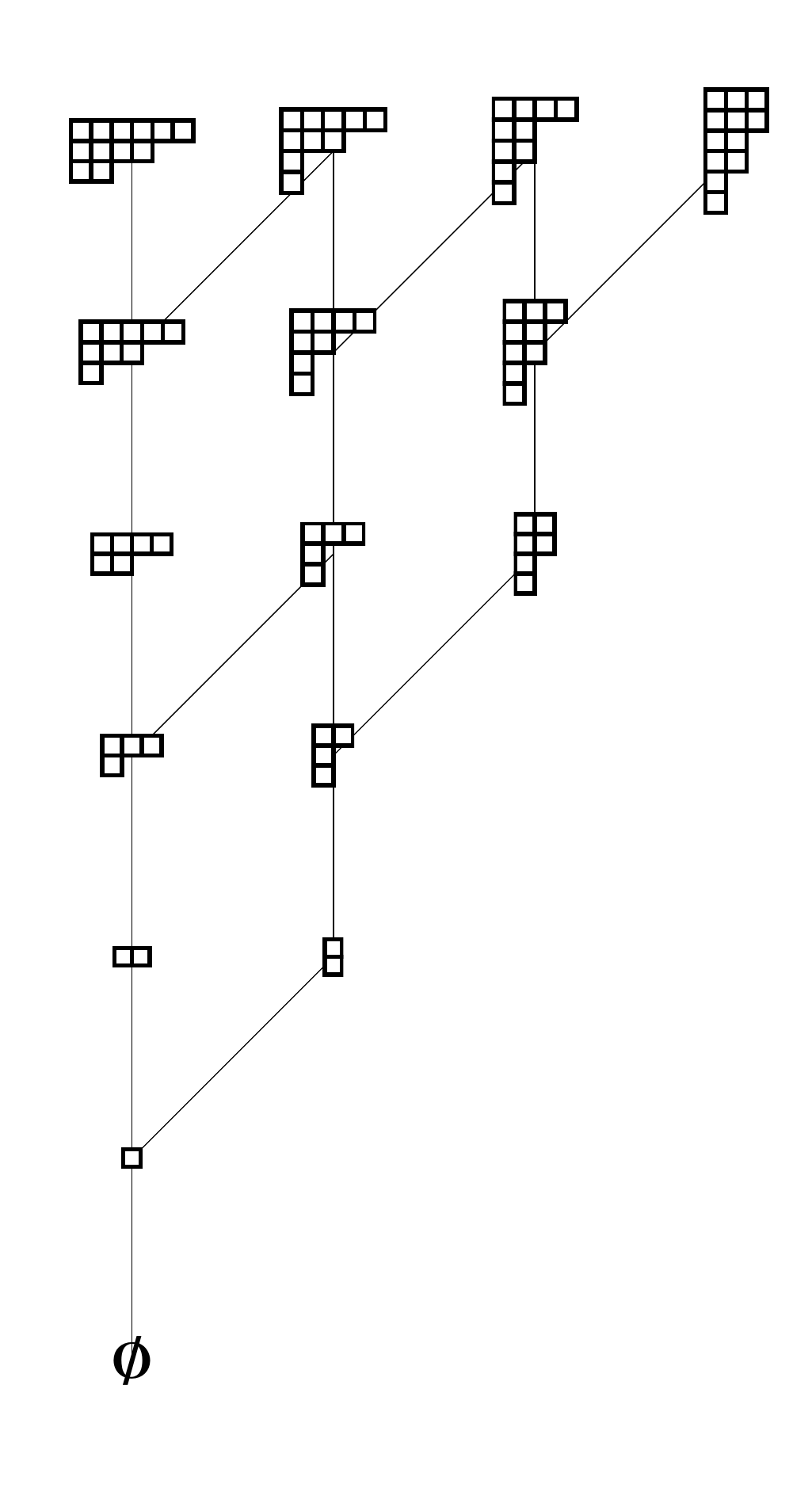}\end{gathered} &
\begin{gathered}\includegraphics[height=2.4in]{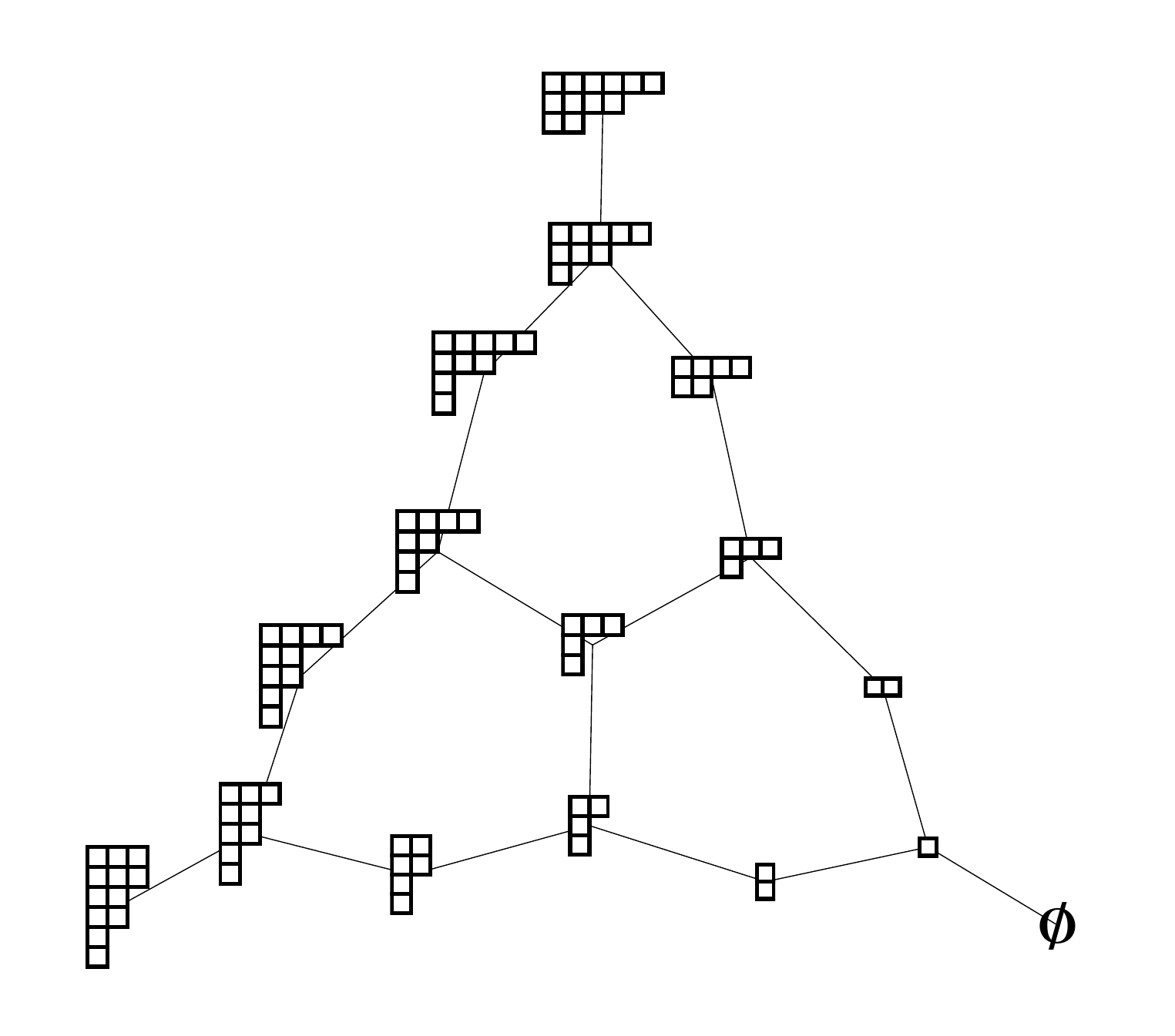}\end{gathered} \\
\end{array}$
\end{center}
\caption{The Hasse diagram of $\mathcal{Y}^2_4$ and its underlying graph.}
\label{fig:y24}
\end{figure}

With these definitions in hand, it is easy to show that $\mathcal{X}_m^k$ and $\mathcal{Y}_m^k$ are isomorphic as posets.  Understanding the cyclic symmetry of $\mathcal{X}_m^k \simeq \mathcal{Y}_m^k \simeq \mathcal{Z}_m^k$ is therefore equivalent to finding the inverse of the map from $\mathcal{X}_m^k$ to $\mathcal{W}_m^k$.  This turned out to be much more difficult than expected, but a proof was finally found at the 2012 Combinatorial Algebra meets Algebraic Combinatorics conference at the Universit\'e du Qu\'ebec \`a Montr\'eal.

\vspace{30pt}
%------------------------------------------------------------------------------------------------------------------------
\section{$\mathcal{Z}_m^k$: A poset of alcoves}
\label{sec:alcores}
%------------------------------------------------------------------------------------------------------------------------

Following M.\ Zabrocki and C.\ Berg in~\cite{berg2011symmetries}, we introduce some geometry and define the poset $\mathcal{Z}^k_m$ as a dilation of the fundamental alcove.

Let $\{\alpha_i\}_{i=1}^k$ be \emph{simple roots} for the $A_k$ root system in the vector space $V$ with a positive definite symmetric bilinear form $\langle \cdot,\cdot \rangle$.  We denote the set of all roots for $A_k$ by $\Phi$.
We write $\alpha_0 = -\sum_{i=1}^k \alpha_i$ for the negative of the \emph{highest root}.%, so that the geometric cyclic symmetry .

For $v \in V$ and $p \in \mathbb{Z}$, define the \emph{hyperplane} $H_{v,p}=\{x \in V | \langle v, x \rangle = p\}.$   The type $A_k$ affine hyperplane arrangement is the set of hyperplanes $\{ H_{\alpha,p} : \alpha \in \Phi \text{ and } p \in \mathbb{Z} \}$.

For $1\leq i\leq k$, let $s_i$ be the \emph{reflection} in the hyperplane $H_{\alpha_i,0}$ and let $s_0$ be the reflection in $H_{-\alpha_0,1}$.  The group generated by $\{s_i\}_{i=0}^k$ is the \emph{affine symmetric group}, which has relations
\begin{align*}
s_i^2 & = 1 \text{ for } 0\leq i \leq k,\\
s_i s_j &= s_j s_i \text{ if } i-j \neq \pm 1,\\
s_i s_{i+1} s_i &= s_{i+1} s_i s_{i+1} \text{ for } 0\leq i \leq k.
\end{align*}

The \emph{dominant chamber} is the region $\{ x : \langle \alpha_i, x \rangle \geq 0 \text{ for } 1\leq i \leq k \}$.  The connected components of $V / \bigcup_{\alpha \in \Phi, p \in \mathbb{Z}} H_{\alpha,p}$ are called \emph{alcoves}.  The \emph{fundamental alcove} is the intersection of the dominant chamber with $\{ x : \langle -\alpha_0, x \rangle \leq 1 \}$.

\begin{definition} \rm
Define a partial order on the alcoves in the dominant chamber by taking the fundamental alcove to be minimal and letting two alcoves have a covering relation when they share a bounding hyperplane (this hyperplane is called a \emph{wall}).
\end{definition}

\begin{definition}[\cite{berg2011symmetries}]
Let $\mathcal{Z}^k_m$ consist of the alcoves contained in the $m$-fold dilation of the fundamental alcove.
\end{definition}

In other words, $\mathcal{Z}^k_m$ contains those alcoves in the dominant chamber bounded by the hyperplane $H_{-\alpha_0,m}$.  Note that $\mathcal{Z}^k_m$ is a dilation of the fundamental alcove, which geometrically exhibits a cyclic symmetry of order $(k+1)$ obtained by permuting $\alpha_1, \alpha_2, \ldots, \alpha_0$.

\vspace{30pt}
%------------------------------------------------------------------------------------------------------------------------
\section{$\mathcal{Y}_m^k$: A subposet of $k$-Young's lattice}
\label{sec:posety}
%------------------------------------------------------------------------------------------------------------------------

We conclude our summary of M.\ Zabrocki and C.\ Berg's results in~\cite{berg2011symmetries}.  We encode the geometry of the dominant chamber using partition cores and state the result that $\mathcal{Z}^k_m$ restricts to a poset $\mathcal{Y}^k_m$ containing $(k+1)$-cores lying below certain stacks of rectangles.  We do not know in general of a simple combinatorial rule to perform the inherited geometric cyclic action explicitly on the cores.%---instead, we content ourselves that the cyclic symmetry of the graph of $\mathcal{Y}^k_m$ is inherited from the geometric cyclic symmetry of $\mathcal{Z}^k_m$.

%of $\mathcal{Z}^k_m$ using partition cores.  We restrict 
%, and give their interpretation as a   
%Using Theorem~\ref{thm:lascoux}, there is a bijection between $\mathcal{Y}^k_m$ and   The maximal elements can be shown to be the $R_{k,\{i_1,i_2,\ldots,i_{m-1}\}}$.

We interpret the poset of the alcoves in the dominant chamber as a poset on certain partition cores.

\begin{definition} \rm
The \emph{hook length} of a box in the Ferrers diagram (in English notation) of a partition is one plus the number of boxes to the right in the same row plus the number of boxes below and in the same column.  A \emph{$(k+1)$-core} is a partition with no hook length of size $(k+1)$.
\end{definition}

We label the $(i,j)$th box of the Ferrers diagram of a $(k+1)$-core $\lambda$ by its \emph{content} $(j-i) \mod (k+1)$.  We define an action of the affine symmetric group on $(k+1)$-cores $\lambda$ by letting $s_i \lambda$ (for $0 \leq i \leq k$) be the unique $(k+1)$-core that differs from $\lambda$ only by boxes with content $i$.

\begin{definition} \rm
Define a partial order on the set of $(k+1)$-cores by fixing
the covering relations: $\lambda$ covers $\mu$ if 
$|\lambda|>|\mu|$ and $\lambda = s_i \mu$ for some $i$.  
\end{definition}

\begin{theorem}[~\cite{lascoux2001ordering}]
\label{thm:lascoux}
	There is a poset isomorphism between alcoves in the dominant chamber and the poset of $(k+1)$-cores.
\end{theorem}

One may compute this bijection by observing that the action of the affine symmetric group on the cores parallels the action of reflecting an alcove across one of its bounding hyperplanes.

\begin{figure}[ht]
\begin{center}
\includegraphics[height=3.3in]{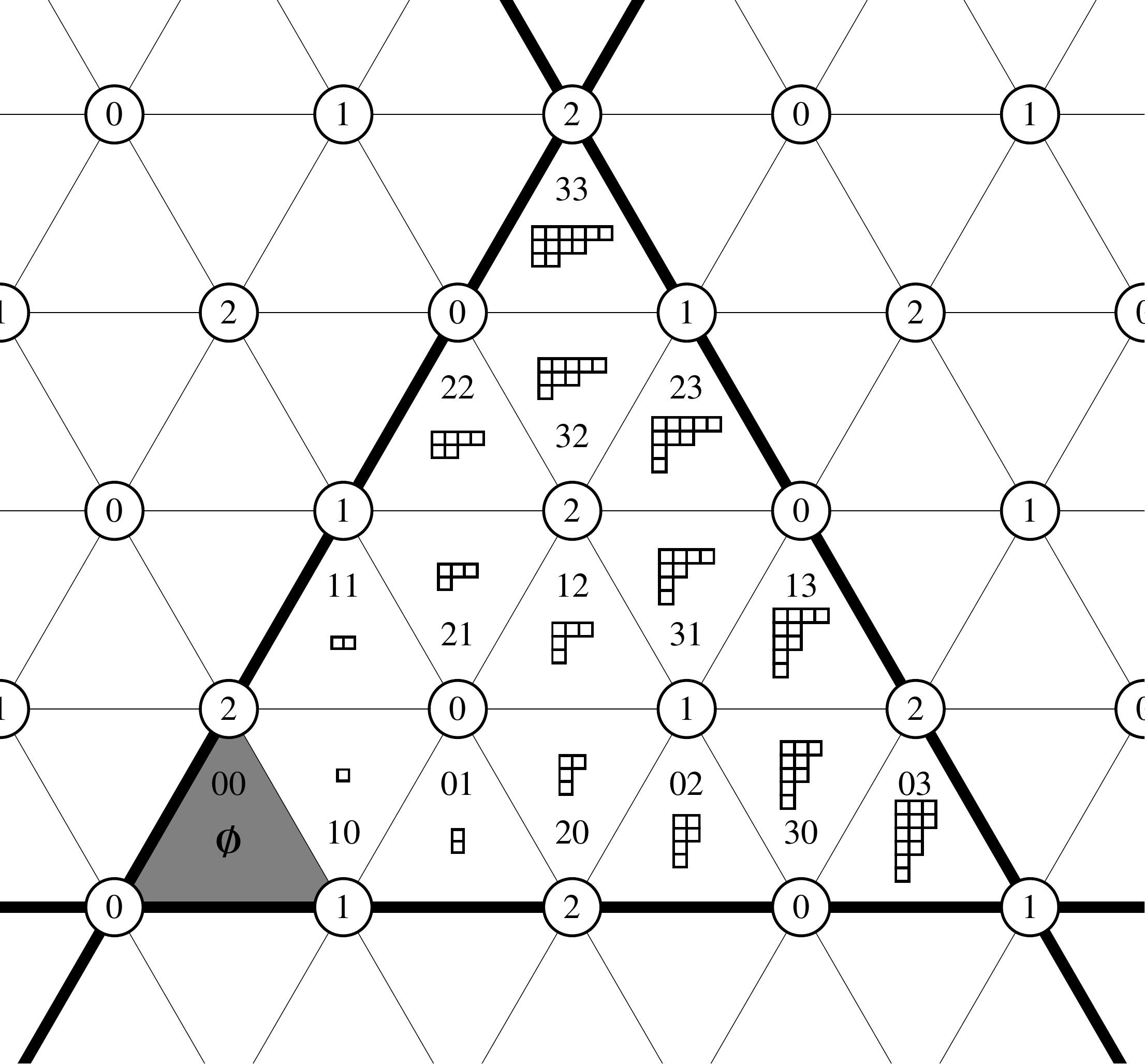}
\end{center}
\caption{The fundamental alcove is shaded gray.  The thick black lines represent the hyperplanes $H_{\alpha_1,0}$, $H_{\alpha_2,0}$, and $H_{-\alpha_0,4}$.  The alcoves within this bounded region are filled with their corresponding cores in $\mathcal{Y}_4^2$ and words in $\mathcal{X}_4^2$.}
\label{fig:geometry}
\end{figure}

Using Theorem~\ref{thm:lascoux}, C.\ Berg and M.\ Zabrocki specified the maximal elements of $\mathcal{Z}^k_m$ in terms of $(k+1)$-cores~\cite{berg3610expansion}.

\begin{definition} \rm
Let $R_{k,i} = (i^{k-i+1})$ be the rectangular partition with $(k-i+1)$ parts of size $i$.  For fixed $k,m$ and for $i_1 \geq i_2 \geq \cdots \geq i_{m-1}$, let $R_{k,\{i_1,i_2,\ldots,i_{m-1}\}}$ be the Ferrers diagram obtained by placing the Ferrers diagram of $R_{k,i_{j+1}}$ at the lower left corner of the Ferrers diagram of $R_{k,i_{j}}$.
\end{definition}

\begin{example}
For $m=4$ and $k=2$, we have

\ytableausetup{boxsize=0.5em}
$$R_{\{1,1,1\}}=\ydiagram[*(white)]
{2+1,2+1,1+1,1+1,1,1}
*[*(gray)]{3,3,2,2,1,1},
R_{\{2,1,1\}}=\ydiagram[*(white)]
{2+2,1+1,1+1,1,1}
*[*(gray)]{4,2,2,1,1},
R_{\{2,2,1\}}=\ydiagram[*(white)]
{3+2,1+2,1,1}
*[*(gray)]{5,3,1,1},
R_{\{2,2,2\}}=\ydiagram[*(white)]
{4+2,2+2,2}
*[*(gray)]{6,4,2}.$$
\end{example}

%C.\ Berg and M.\ Zabrocki were 

%led to these definitions by their earlier work on $k$-Schur functions~\cite{berg2011expansions}---specifically, the multiplication rule that arises when one of the $k$-Schur functions is a rectangle.

\begin{definition} \rm
Let $\mathcal{Y}^k_m$ be the subposet of $k$-Young's lattice consisting of all $(k+1)$-cores contained in some $R_{k,\{i_1,i_2,\ldots,i_{m-1}\}}$.
\end{definition}

%Case $m=2$, \cite{suter2002young}; cases $m>2$, 
\begin{theorem}[\cite{berg2011symmetries}]
There is a poset isomorphism between $\mathcal{Y}^k_m$ and $\mathcal{Z}^k_m$
\end{theorem}

Figure~\ref{fig:y24} illustrates the poset $\mathcal{Y}^2_4$ and deforms the graph underlying the poset to reveal the cyclic symmetry of order $3$.

\begin{proof}[Proof sketch]
One can identify the maximal alcoves as those labeled by some $R_{k,\{i_1,i_2,\ldots,i_{m-1}\}}$.  Using Theorem~\ref{thm:lascoux}, we obtain a poset isomorphism between $\mathcal{Y}^k_m$ and $\mathcal{Z}^k_m$.
%To complete the proof, we note that 
\end{proof}

The graph of $\mathcal{Y}^k_m$ therefore inherits a $(k+1)$-fold cyclic symmetry from $\mathcal{Z}^k_m$.  Figure~\ref{fig:geometry} illustrates this geometry.

%talk about Suter? explicit
%For $m=2$, R.\ Suter gave an explicit action on partitions.% act on $\lambda=(\lambda_1,\ldots,\lambda_k)$ by removing the top row of the diagram and prepending a column of height $n-\mu_1-1$ on the left.  It isn't hard to show that this has order $n+1$
%example

\vspace{15pt}

%------------------------------------------------------------------------------------------------------------------------
\section{$\mathcal{X}_m^k$: A poset of words of length $k$}%X_m^k
\label{sec:wordsthatend}
%------------------------------------------------------------------------------------------------------------------------

We give a combinatorial definition of a poset $\mathcal{X}^k_m$ on words, show that this poset is isomorphic to $\mathcal{Y}_m^k$, and realize the symmetry of both posets as an explicit combinatorial action on the words in $\mathcal{X}^k_m$.

\begin{definition} \rm
Let $\mathcal{X}^k_m$ be the poset of words $\mathbf{x}$ of length $k$ on $\mathbb{Z}/m\mathbb{Z}$ with the partial order induced by the following covering relations:
\begin{enumerate}
	\item For $a<m-1$, $\mathbf{y} a \lessdot (a+1) \mathbf{y}$, where 
$\mathbf y$ is a string of length $k-1$ on the alphabet 
$\mathbb{Z}/m\mathbb Z$.  
	\item For $b<a$, $\mathbf{y} ab \mathbf{z} \lessdot \mathbf{y} ba \mathbf{z}$, where $\mathbf y$ and $\mathbf z$ are two strings on the 
alphabet $\mathbb Z/m\mathbb Z$, with the total length of the two strings
being $k-1$.   
\end{enumerate}
\end{definition}

Figure~\ref{fig:x24} illustrates the poset $\mathcal{X}^2_4$.
Given a word $\mathbf{x}$ on $\mathbb{Z}/m\mathbb{Z}$, we write $(\mathbf{x}-i)$ to denote the word obtained by subtracting $i$ from each letter of $\mathbf{x}$ (mod $m$).

\begin{theorem}
The graph of $\mathcal{X}^k_m$ has $(k+1)$-fold cyclic symmetry.
\end{theorem}

\begin{proof}

We will define a cyclic action of order $(k+1)$ that is a graph isomorphism.

\begin{definition} \rm
 Given a word $\mathbf{x} \in \mathcal{X}^k_m$, form the \emph{extended word} of length $(k+1)m$ $$\overline{\mathbf{x}} = (\mathbf{x})(m-1)(\mathbf{x}-1)(m-2)\ldots(\mathbf{x}-m+1)(0).$$   That is, for any $1\leq i \leq (k+1)$, the entries in positions $i, (k+1)+i, 2(k+1)+i, \ldots (m-1)(k+1)+i$ are cyclically decreasing by $1$.  Let $\phi(\overline{\mathbf{x}})$ be defined by cyclically rotating $\overline{\mathbf{x}}$ left so that its leftmost $0$ appears as its rightmost character.  This induces an action $\phi$ on $\mathcal{X}^k_m$ by restricting the resulting word to its first $k$ letters.
\end{definition}
An example is given in Figure~\ref{fig:orbit}.

\begin{figure}[hb]

\begin{center}
\begin{tabular}{|ccc||ccc|} \hline
\multicolumn{3}{|c||}{Orbits of extended words in $\mathcal{X}^4_2$} & \multicolumn{3}{c|}{Orbits of $\mathcal{W}^4_2$}\\ \hline
$\mathbf{00}3 \text{  } 332 \text{  } 221 \text{  } 110$ & $\mathbf{03}3 \text{  } 322 \text{  } 211 \text{  } 100$ & $\mathbf{33}3 \text{  } 222 \text{  } 111 \text{  } 000$ & 003 & 030 & 300 \\ \hline
$\mathbf{10}3 \text{  } 032 \text{  } 321 \text{  } 210$ & $\mathbf{30}3 \text{  } 232 \text{  } 121 \text{  } 010$ & $\mathbf{32}3 \text{  } 212 \text{  } 101 \text{  } 030$ & 133 & 331 & 313 \\ \hline
$\mathbf{01}3 \text{  } 302 \text{  } 231 \text{  } 120$ & $\mathbf{13}3 \text{  } 022 \text{  } 311 \text{  } 200$ & $\mathbf{22}3 \text{  } 112 \text{  } 001 \text{  } 330$ & 012 & 120 & 201 \\ \hline
$\mathbf{11}3 \text{  } 002 \text{  } 331 \text{  } 220$ & $\mathbf{02}3 \text{  } 312 \text{  } 201 \text{  } 130$ & $\mathbf{23}3 \text{  } 122 \text{  } 011 \text{  } 300$ & 102 & 021 & 210 \\ \hline
$\mathbf{21}3 \text{  } 102 \text{  } 031 \text{  } 320$ & $\mathbf{20}3 \text{  } 132 \text{  } 021 \text{  } 310$ & $\mathbf{31}3 \text{  } 202 \text{  } 131 \text{  } 020$ & 223 & 232 & 322 \\ \hline
 & $\mathbf{12}3 \text{  } 012 \text{  } 301 \text{  } 230$ & & & 111 & \\ \hline
\end{tabular}
\end{center}
%The corresponding orbits of $\mathbf{x}$.
%\begin{tabular}{ccc}
%$00$ & $03$ & $33$ \\
%$10$ & $30$ & $32$ \\
%$01$ & $13$ & $22$ \\
%$11$ & $02$ & $23$ \\
%$21$ & $20$ & $31$ \\
%$12$ & & \\
%\end{tabular}
\caption{When $m=4$ and $k=2$, there are five orbits of size 3 and one orbit of size 1.}  %The rows of the figure are the words in the orbit of $\overline{20}$, with the corresponding words in $\mathcal{X}_4^2$ in bold.}
\label{fig:orbit}
\end{figure}

Observe that $\phi$ is a cyclic action of order $(k+1)$, since there are $(k+1)$ zeros in $\overline{\mathbf{x}}$.  It is now a tedious check to show that $\phi$ takes edges to edges.  Note that $\phi$ is not a poset isomorphism: it reverses the orientation of some edges.  
%We now show that $\phi$ takes edges to edges by analyzing four cases.

\vspace{20pt}

$\bullet$ \textbf{Case 1 (a)} (An edge of Type (1), where the position of the leftmost zero in $\overline{\mathbf{y} a}$ is not $j(k+1)+k$):
%$a-j$ is not the leftmost zero in $\overline{\mathbf{y}a}$ for some $j\geq 0$):

Let the position $j(k+1)+i$ be the leftmost zero in $\overline{\mathbf{y}a}$, so that $y_{i}-j=0$.% be the leftmost zero. 

Applying $\phi$ to $\mathbf{y_L} y_{i} \mathbf{y_R} a \lessdot (a+1)\mathbf{y_L} y_{i} \mathbf{y_R}$ gives us $$(\mathbf{y_R}-j) (a-j) (m-1-j) (\mathbf{y_L}-j-1) \gtrdot (\mathbf{y_R}-j) (m-1-j)(a-j) (\mathbf{y_L}-j-1),$$

%$$\mathbf{y} a \lessdot (a+1) \mathbf{y}$$ gives us
%\begin{align*}
%\mathbf{y_L} y_{i+1} \mathbf{y_R} a (m-1) \ldots & (\mathbf{y_L}-j)0 (\mathbf{y_R}-j) (a-j)(m-1-j) \ldots \\
%\lessdot (a+1) \mathbf{y_L} y_{i+1} \mathbf{y_R} (m-1) \ldots & (a+1-j) (\mathbf{y_L}-j)0(\mathbf{y_R}-j) (m-1-j) \ldots.
%\end{align*}

%The action $\phi$ gives us
%\begin{align*}
%& (\mathbf{y_R}-j) (a-j) (m-1-j) (\mathbf{y_L}-j-1)\\
%\gtrdot & (\mathbf{y_R}-j) (m-1-j)(a-j) (\mathbf{y_L}-j-1),
%\end{align*}
which is an edge of Type 2.

$\bullet$ \textbf{Case 1 (b)} (An edge of Type (1), where the position of the leftmost zero in $\overline{\mathbf{y} a}$ is $j(k+1)+k$):
%$a-j$ is the leftmost zero in $\overline{\mathbf{y}a}$ for some $j\geq 0$):

%If the subword $\mathbf{y}$ in the relation $$\mathbf{y} 0 \lessdot 1 \mathbf{y}$$ contains a 0, then the analysis from Case 1 (a) holds.  Otherwise
Let the position $j(k+1)+k$ be the leftmost zero in $\overline{\mathbf{y}a}$, so that $a-j=0$.

Applying $\phi$ to $\mathbf{y} a \lessdot (a+1) \mathbf{y}$ gives us $$(m-1-j)(\mathbf{y}-1-j) \gtrdot (\mathbf{y}-1-j) (m-2-j),$$ which is an edge of Type 1.

\vspace{20pt}

For an edge of Type (2), let the positions of $a$ and $b$ in $\mathbf{y} ab \mathbf{z}$ be given by $i$ and $i+1$.

$\bullet$  \textbf{Case 2 (a)} (An edge of Type (2), where the position of the leftmost zero in $\overline{\mathbf{y} ab \mathbf{z}}$ is not $a(k+1)+i$ or $b(k+1)+i+1$):

It is clear that the relation $$\mathbf{y}ab\mathbf{z} \lessdot \mathbf{y}ba\mathbf{z}$$ will translate to another edge of Type (2).

$\bullet$ \textbf{Case 2 (b)} (An edge of Type (2), where the position of the leftmost zero in $\overline{\mathbf{y} ab \mathbf{z}}$ is $a(k+1)+i$ or $b(k+1)+i+1$):

If the leftmost zero is in position $a(k+1)+i$, applying $\phi$ to $\mathbf{y_L} a b \mathbf{y_R} \lessdot \mathbf{y_L} b a \mathbf{y_R}$ gives us $$(b-a) (\mathbf{y_R}-a) (\mathbf{y_L}-a-1) \gtrdot (\mathbf{y_R}-a) (\mathbf{y_L}-a-1) (b-a-1),$$ which is an edge of Type (1).

If the leftmost zero is instead in position $b(k+1)+i+1$, applying $\phi$ to $\mathbf{y_L} a b \mathbf{y_R} \lessdot \mathbf{y_L} b a \mathbf{y_R}$ gives us $$(\mathbf{y_R}-b) (\mathbf{y_L}-b-1) (a-b-1) \lessdot (a-b)(\mathbf{y_R}-b) (\mathbf{y_L}-b-1),$$ which is again an edge of Type (1).

%We have the relation $$\mathbf{y}a0\mathbf{z} \lessdot \mathbf{y}0a\mathbf{z}.$$  If the subword $\mathbf{y}$ contains a zero, $\phi$ gives us another edge of Type 2.  Otherwise, $\phi$ gives us the relation $$\mathbf{z}m(\mathbf{y}-1)(a-1) \lessdot a\mathbf{z}m(\mathbf{y}-1),$$ which is an edge of Type 1.

%\textbf{Case 2 (a)} (An edge of Type (2), with $b\neq 0$):

%It is clear that the relation $$\mathbf{y}ab\mathbf{z} \lessdot \mathbf{y}ba\mathbf{z}$$ will translate to another edge of Type 2.

%\textbf{Case 2 (b)} (An edge of Type (2), with $b=0$):

%We have the relation $$\mathbf{y}a0\mathbf{z} \lessdot \mathbf{y}0a\mathbf{z}.$$  If the subword $\mathbf{y}$ contains a zero, $\phi$ gives us another edge of Type 2.  Otherwise, $\phi$ gives us the relation $$\mathbf{z}m(\mathbf{y}-1)(a-1) \lessdot a\mathbf{z}m(\mathbf{y}-1),$$ which is an edge of Type 1.
\end{proof}

\begin{theorem}[Case $m=2$, \cite{williamsbijactions}]
\label{thm:coretocombo}
There is a poset isomorphism between $\mathcal{X}_m^k$ and $\mathcal{Y}_m^k$.  The geometric symmetry of the underlying graph is realized by the cyclic action $\phi$.
\end{theorem}

To efficiently define this isomorphism, we first recall the abacus model.

\begin{definition} \rm
Given a partition, we read off the path formed by the boundary of its Ferrers diagram from top right to bottom left as a \emph{boundary word}, where a $1$ records a step left and a $0$ records a step down.

If we are further given a positive integer $(k+1)$, we may form an \emph{abacus display} with $(k+1)$ runners (labeled $0,1,\ldots,k$) by breaking the boundary word into consecutive runs of length $(k+1)$ and then stacking them.  Given a $(k+1)$-core contained in $R_{k,\{i_1,i_2,\ldots,i_{m-1}\}}$, we will choose the particular representative of it as an abacus display with $(k+1)$ runners by forcing the first zero of the boundary word to lie in the leftmost column.
\end{definition}

%The advantage of this display is that one may immediately see which partitions are $(k+1)$-cores---they are exactly those with columns flush to the top.

%example

The first three rows of Figure~\ref{fig:posetiso} illustrate the construction of the abacus representative from a partition.  We now proceed with the proof of Theorem~\ref{thm:coretocombo}.

\begin{proof}

%\begin{enumerate}
%	\item Take the finite part of the binary word between the first $0$ and the last $1$,
%	\item Pad the boundary word with zeros on the right to make it of length $(m-1)(k+1)$,
%	\item Break the result into $(m-1)$ consecutive runs, each of length $(k+1)$.
%\end{enumerate}
Since the partition is a core, the columns will be flush~\cite{james1984representation}. %[FINISH]: Reference
  We may therefore recover this display from the word $x_1 \cdots x_k$, where $x_i$ counts the number of ones in the $i$th column (occurring after the first zero).
%\begin{enumerate}
%	\item Add these smaller words together component-wise, and
%	\item Suppress the leading zero.
%\end{enumerate}
Finally, note that since the core was contained in $R_{k,\{i_1,i_2,\ldots,i_{m-1}\}}$, it will have at most $m-1$ rows.  We have therefore defined a bijection from $\mathcal Y^k_m$ to words of length $k$ on $\mathbb{Z}/m\mathbb{Z}$.

To complete the proof, we show that the poset structures are the same.  The empty partition is mapped to the word of all zeros.  Adding a box to the Ferrers diagram of a general partition changes a consecutive pair $\ldots 10 \ldots$ in the boundary word to the pair $\ldots 01 \ldots$.  Adding all possible boxes with a specific content to a core simultaneously applies this change to all such $\ldots10\ldots$ pairs in the boundary word that lie $(k+1)$ positions apart.  Converting to the abacus model stacks entries that differ by $(k+1)$ positions.

If we do not add any boxes to the first row, then the $\ldots10\ldots$ pairs are not split between the first and last column.  A covering relation of Type 2 in $\mathcal{Y}_m^k$ therefore corresponds to adding as many boxes as possible with the same content when no boxes are added to the first row.
%In this case, we will have to reorient our abacus model so that the condition that the first column contains only zeroes is satisfied.

On the other hand, adding boxes including one on the first row corresponds to the $\ldots10\ldots$ pairs being split between the first and last column.  In this case, the last column will be emptied of all its ones and the additional one directly before the first zero will be introduced.  This corresponds to an edge of Type 1.

%so that we can just cycle the abacus model over by one

%  To see that  while an edge of Type 1 corresponds to such a move when boxes are added to the first row.%, while preserving the property that the partition is a $(k+1)$-core.
\end{proof}

When $m=4$ and $k=2$, this bijection is illustrated for a single orbit in Figure~\ref{fig:posetiso}.  The full correspondence for this example is given in Figure~\ref{fig:geometry}.

%[FINISH]: insert example of this
\begin{figure}[ht]
\begin{tabular}{|c|c|c|c|}
\hline
3-core in $\mathcal{Y}_4^2$ & $\ydiagram{3,1}$ & $\ydiagram{2,1,1}$ & $\ydiagram{4,2,1,1}$  \\ \hline
Boundary word & $011|010|000$ & $010|010|000$ & $011|010|010$  \\ \hline
Abacus display & $\begin{gathered}\begin{matrix}
  0 & 1 & 1 \\
  0 & 1 & 0 \\
  0 & 0 & 0
 \end{matrix}\end{gathered}$ & $\begin{gathered}\begin{matrix} 0 & 1 & 0 \\ 0 & 1 & 0 \\ 0 & 0 & 0\end{matrix}\end{gathered}$ & $\begin{gathered}\begin{matrix}
  0 & 1 & 1 \\
  0 & 1 & 0 \\
  0 & 1 & 0
 \end{matrix}\end{gathered}$ \\ \hline
Word in $\mathcal{X}_4^2$ & $21$ & $20$ & $31$ \\ \hline
\end{tabular}
\caption{The poset isomorphism given in Theorem~\ref{thm:coretocombo} from an orbit of $\mathcal{Y}_4^2$ to the corresponding orbit of $\mathcal{X}_4^2$.}
\label{fig:posetiso}
\end{figure}

\begin{remark} \rm
In the case $m=2$, R.\ Suter gave an explicit description of the action on cores~\cite{suter2002young}.  It was proved in~\cite{williamsbijactions} that this is the same as $\phi$.
\end{remark}

\vspace{15pt}
%------------------------------------------------------------------------------------------------------------------------
\section{$\mathcal{W}_m^k$: Words of length $(k+1)$ that sum to $(m-1)$}
\label{sec:wordsthatsum}
%------------------------------------------------------------------------------------------------------------------------

We define the cyclic sieving phenomenon, give a set of words $\mathcal{W}_m^k$ under rotation that exhibit it, and give the forward direction of an equivariant bijection from $\mathcal{X}_m^k$ to $\mathcal{W}_m^k$.  This equivariant bijection exchanges the natural poset structure on $\mathcal{X}_m^k$ for a natural cyclic action on $\mathcal{W}_m^k$.

The cyclic sieving phenomenon (CSP) was introduced by V.\ Reiner, D.\ Stanton, and D.\ White~\cite{reiner2004cyclic} as a generalization of J.\ Stembridge's $q=-1$ phenomenon~\cite{stembridge1994some}.%FINISH REFERENCE

\begin{definition}[\cite{reiner2004cyclic}] \rm
	Let $X$ be a finite set, $X(q)$ a generating function for $X$, and $C$ a cyclic group acting on $X$.  Then the triple $(X, X(q), C)$ exhibits the CSP if for $c \in C,$ $$X(\omega(c)) = \left| \{ x \in X : c(x) = x \} \right|,$$ where $\omega: C \to \mathbb{C}$ is an isomorphism of $C$ with the $n$th roots of unity. 
\end{definition}

In other words, a set exhibits the CSP if we can obtain information about its orbit structure under a cyclic action by evaluating a polynomial at a root of unity.

\begin{definition} \rm
Let $\mathcal{W}_{m}^{k}$ be the set of all words on $\mathbb{Z}/m\mathbb{Z}$ of length $(k+1)$ with sum equal to $(m-1) \pmod m$.  Let $$\mathcal{W}_{m}^{k}(q) = \prod_{i=1}^{k} \frac{1-q^{m i}}{1-q^i}$$ be a generating function for $\mathcal{W}_{m}^{k}$.
\end{definition}

We let $C_{k+1}$ act by left rotation, sending $w_1 w_2 \ldots w_{k+1}$ to $w_2 \ldots w_{k+1} w_{1}$.  The following lemma is an easy exercise by direct computation.%, using L'H\^{o}pital's Rule.

\begin{lemma}
$(\mathcal{W}_{m}^{k},\mathcal{W}_{m}^{k}(q),C_{k+1})$ exhibits the CSP.
\end{lemma}

To prove the corresponding statement for $\mathcal{X}_m^k$, we define an equivariant bijection with $\mathcal{W}_m^k$.

{
\renewcommand{\thetheorem}{\ref{thm:forward}}
\begin{theorem}
	There is an equivariant bijection $w$ from $\mathcal{X}_{m}^k$ under $\phi$ to $\mathcal{W}_{m}^k$ under left rotation.
\end{theorem}
\addtocounter{theorem}{-1}
}

\begin{proof}
One direction is easy.  We take the first letter of each word from an orbit of $\mathcal{X}_{m}^k$ under $\phi$, and concatenate these letters into a single word, cyclically repeated to make the resulting word of length $(k+1)$. (This trick is called a \emph{bijaction}~\cite{williamsbijactions}.)

We now calculate the sum of the entries in $w(\mathbf x)$.  By definition of the cyclic action, each $x_i$ in the word $\mathbf{x}(m-1)$ will occur as the first letter of a word in the orbit of $\mathbf{x}$ when a translate of the preceding letter $x_{i-1}$ is equal to zero.  Writing $\mathbf{x}(m-1) = \mathbf{x_L} x_{i-1} x_i \mathbf{x_R}$, we form the extended word of length $(k+1)m$ $$\mathbf{x_L} x_{i-1} x_i \mathbf{x_R} \ldots (\mathbf{x_L}-x_{i-1}) 0 (x_i -x_{i-1}) (\mathbf{x_R}-x_{i-1}) \ldots ,$$
%(\mathbf{x_L}-(m-1)) (x_{i-1}-(m-1)) (x_i -(m-1)) (\mathbf{x_R}-(m-1)),$$
from which we conclude that each $x_i$ in $\mathbf{x}(m-1)$ is counted as $(x_i-x_{i-1}) \pmod m$.  Summing over all $i$ (and so disregarding the order in which the letters appear), we have a telescoping sum which leaves only the last letter of $\mathbf{x}(m-1)$.  Therefore, the sum of the word $w(\mathbf{x})$ is $(m-1) \pmod m$.%, which $m-1$.

%The sum is $(m-1) \pmod m$ because a translate of each $x_i$ in the word $\mathbf{x}(m-1)$ occurs as the first letter of a word in the orbit when $x_{i-1}+j=0 \pmod m$ as $(x_i+j) \pmod m$.  So $j=-x_{i-1} \pmod m$, and $x_i$ occurs as $(x_i-x_{i-1}) \pmod m$.  Summing over all $i$, we have a telescoping sum which leaves only the last letter, which is $m-1$.

By construction, $\phi$ maps to left rotation.  The inverse $w^{-1}$ will be constructed over the next two sections.
\end{proof}

This bijection is illustrated in Figure~\ref{fig:orbit}.  As an immediate corollary, we understand the orbit structures of $\mathcal{X}_{m}^k$, $\mathcal{Y}_{m}^k$, and $\mathcal{Z}_{m}^k$ under their cyclic actions.

\begin{corollary}
\label{cor:cspthm}
$(\mathcal{X}_{m}^k \simeq \mathcal{Y}_{m}^k \simeq \mathcal{Z}_{m}^k,\mathcal{W}_{m}^{k}(q),\langle \phi \rangle)$ exhibits the CSP.
\end{corollary}

\vspace{15pt}
%------------------------------------------------------------------------------------------------------------------------
\section{Dendrodistinctivity}
\label{sec:dendro}
%------------------------------------------------------------------------------------------------------------------------

We prove that a generalization of the map $w$ in Theorem~\ref{thm:forward} is a bijection.  Section~\ref{sec:proof} will show how to specialize this construction to $w$ and $w^{-1}$.

\begin{definition}
Let $\mathbb{W}_m^{k}$ be the set of words of length $(k+1)$ on $\mathbb{Z}/m\mathbb{Z}$.
\end{definition}

\begin{definition} \rm
Let $\mathbf{w}=w_1 w_2 \ldots w_{k+1} \in \mathbb{W}_m^k$ and let $\mathbf{b}=(b_0 \leq b_1 \leq \cdots \leq b_{m-2})$ be an $(m-1)$-tuple with entries in $\{0,1,2,\ldots,k+1\}$.  Define a \emph{partitioned word} $$(\mathbf{w},\mathbf{b})=w_1 w_2 \ldots w_{b_0}|w_{b_0+1} w_{b_0+2} \ldots w_{b_1}|\cdots|w_{b_{m-2}+1} w_{b_{m-2}+2} \ldots w_{k+1}$$ to be a partition of $\mathbf{w}$ into $m$ connected blocks, where $\mathbf{b}$ specifies where the dividers are placed.  We denote the set of all partitioned words on $\mathbb{Z}/m\mathbb{Z}$ with $\mathbf{w}$ of length $k+1$ by $(\mathbb{W}_m^k)^*$.
\end{definition}

Our map is defined in two parts.  Algorithm~\ref{map:forward} defines a map $p: \mathbb{W}_m^k \to (\mathbb{W}_m^k)^*$.  Write $\mathbb I_m^k$ for the image
of $p$ in $(\mathbb{W}_m^k)^*$.  
Let $f: (\mathbb{W}_m^k)^* \to \mathbb{W}_m^k$ be the map on partitioned words
which forgets the partition.  Composing, we obtain a map $f \circ p$ from $\mathbb{W}_m^k$ to itself.

The inverse map $(f \circ p)^{-1}$ is also defined in two parts.  Algorithm~\ref{map:inverse} defines a map $q$ which takes 
$\mathbb{I}_m^k$ to $\mathbb W_m^k$.  This map $q$ is the inverse of $p$.
We also define a map 
$g: \mathbb W_m^k \to \mathbb I_m^k$ which is the inverse of $f$: it reconstructs
the partition forgotten by $f$.  

$$\xymatrix{
 \mathbb{W}_m^k \ar@<.7ex> [r]^{p} & \mathbb{I}_m^k \ar@<.7ex>[rrrr]^{f=\text{Forget partition}} \ar@<.7ex> [l]^{q} & & & & \ar@<.7ex>[llll]^{g=\text{Reconstruct partition}} \mathbb{W}_m^k 
% \\
% \mathcal{T}_m^*
}$$

%\vspace{20pt}

%develop a theory of partitioned words and show that a certain class of partitionings are uniquely determined by the underlying words.

% $\mathbf{b}_{\sigma+1},\mathbf{b}_{\sigma+2}, \ldots, \mathbf{b}_\sigma$ with rightmost positions of the blocks indexed by $0\leq b_1\leq b_2 \leq \ldots \leq b_{m}=k+1$.  We denote the set of all partitioned words on $\mathbb{Z}/m\mathbb{Z}$ of length $k+1$ by $(\mathbb{W}_m^k)^*$.

%Given a partitioned word $\mathbf{b(w)}$ and an index $i$, we find the position of the block to which the letter $w_i$ belongs by defining $\block(\mathbf{b(w)},i)=t$, where $b_{t-1}+1 \leq i \leq b_t$ (and we take $b_0=0$).
	%=
	%(\mathbf{b}_{\sigma-m+1},\mathbf{b}_{\sigma-m+2}, \ldots, \mathbf{b}_\sigma)

Before we give the definitions of $p$ and $q$, we need some additional notation for partitioned words.

\begin{definition} \rm
Given a partitioned word $(\mathbf{w},\mathbf{b})$, let $\sigma = \sum_{i=1}^{k+1} w_i \pmod m$.  We label the blocks from left to right as $(\mathbf{w},\mathbf{b})_{\sigma+1},(\mathbf{w},\mathbf{b})_{\sigma+2}, \ldots, (\mathbf{w},\mathbf{b})_\sigma$.
\end{definition}

Algorithm~\ref{map:forward} now defines $p$.

%We define $\mathbb I_m^k$ to be the image of $p$ in $(\mathbb{W}_m^k)^*$.  

\begin{algorithm}[ht]
\KwIn{A word $\mathbf{x} = x_1 x_2 \ldots x_{k+1}$.}
\KwOut{A partitioned word $(\mathbf{w},\mathbf{b})$ with sum $\sigma = x_{k+1} \pmod m$.}
$x_0:=0$\;
%$t_0:=0$\;
Initialize $(\mathbf{w},\mathbf{b})$ as $m$ empty blocks labeled from left to right as $(\mathbf{w},\mathbf{b})_{\sigma+1}:=(\mathbf{w},\mathbf{b})_{\sigma+2}:= \cdots := (\mathbf{w},\mathbf{b})_\sigma := \emptyset$ \;
\For{$i=1$ \KwTo $k+1$}{
	Set $t_i:=x_{i}-x_{i-1}$\;
	Insert $t_i$ as the rightmost letter in block $(\mathbf{w},\mathbf{b})_{x_{i-1}}$\;
}
Return $(\mathbf{w},\mathbf{b})$\;
\label{map:forward}
\caption{$p: \mathbb{W}_m^k \to (\mathbb{W}_m^k)^*$.}
\end{algorithm}

On the other hand, given a partitioned word $(\mathbf{w},\mathbf{b})$, we may perform Algorithm~\ref{map:inverse}.

\begin{algorithm}[ht]
\KwIn{A partitioned word $(\mathbf{w},\mathbf{b})$.}
\KwOut{A pair $(\mathbf{x}, (\mathbf{w'},\mathbf{b'}))$, where $\mathbf{w'}$ is a subword of $\mathbf{w}$.}
$t:=0$\;
\For{$i=1$ \KwTo $k+2$}{
	\If{$(\mathbf{w},\mathbf{b})_t \neq \emptyset$}{
		Let $v_{i}$ be the leftmost letter in $(\mathbf{w},\mathbf{b})_t$\;
		Delete $v_i$ from $(\mathbf{w},\mathbf{b})_t$\;
		Update the current block to $t:=t+v_i \pmod m$\;
		Set $x_i:=t$;
	}
	\Else{Return $(x_1 x_2 \ldots x_{i-1}, (\mathbf{w},\mathbf{b}))$\;}
}
\caption{$(\mathbb{W}_m^k)^* \to \mathbb{W}_m^k \times (\mathbb{W}_m^{k-i+1})^*$}
%A generalized version of $(w^*)^{-1}$, using partitioned words.}
\label{map:inverse}
\end{algorithm}

\begin{definition} \rm
We say that Algorithm~\ref{map:inverse} \emph{succeeds} on $(\mathbf{w},\mathbf{b})$ if the length of the output $\mathbf{x}$ is the same as the length of $\mathbf{w}$ (so that the output $\mathbf{w'}$ is empty).  In this case, $x_{k+1}=\sigma$ and we call $(\mathbf{w},\mathbf{b})$ (or just $\mathbf{b}$ when $\mathbf{w}$ is understood) a \emph{successful} partition of $\mathbf{w}$.
\end{definition}

Figure~\ref{fig:algoex} illustrates Algorithm~\ref{map:inverse} applied to a successful partition.

\begin{figure}[ht]
\begin{center}
\begin{tabular}{|c|c|c|c|}
\hline
$i$ & $t$ & $(\mathbf{w},\mathbf{b})$ & $x_1 x_2 \ldots x_{i-1}$ \\ \hline
1 & 0 & $\mathbf{3}|2|1|0302$ & $\cdot$\\ \hline
2 & 3 & $\cdot|2|1|\mathbf{0}302$ & $3$\\ \hline
3 & 3 & $\cdot|2|1|\mathbf{3}02$ & $33$\\ \hline
4 & 2 & $\cdot|2|\mathbf{1}|02$ & $332$\\ \hline
5 & 3 & $\cdot|2|\cdot|\mathbf{0}2$ & $3323$\\ \hline
6 & 3 & $\cdot|2|\cdot|\mathbf{2}$ & $33233$\\ \hline
7 & 1 & $\cdot|\mathbf{2}|\cdot|\cdot$ & $332331$\\ \hline
8 & 3 & $\cdot|\cdot|\cdot|\cdot$ & $3323313$\\ \hline

%2 & 1 & $3|202|\cdot|\mathbf{0}221$ & 1 \\ \hline
%3 & 1 & $3|202|\cdot|\mathbf{2}21$ & 11 \\ \hline
%4 & 3 & $3|\mathbf{2}02|\cdot|21$ & 113 \\ \hline
%5 & 1 & $3|02|\cdot|\mathbf{2}1$ & 1131 \\ \hline
%6 & 3 & $3|\mathbf{0}2|\cdot|1$ & 11313 \\ \hline
%7 & 3 & $3|\mathbf{2}|\cdot|1$ & 113133 \\ \hline
%8 & 1 & $3|\cdot|\cdot|\mathbf{1}$ & 1131331 \\ \hline
%9 & 2 & $\mathbf{3}|\cdot|\cdot|\cdot$ & 11313312 \\ \hline
%10 & 1 & $\cdot|\cdot|\cdot|\cdot$ & 113133121 \\ \hline
\end{tabular}
\end{center}
\caption{Algorithm~\ref{map:inverse} applied to a successful partition of the word $\mathbf{w}=3210302$, with $m=4$, $k=6$.}
\label{fig:algoex}
\end{figure}

\begin{lemma} For $\mathbf w \in \mathbb W_m^k$, 
Algorithm~\ref{map:inverse} succeeds when it is applied to $p(\mathbf w)$,
and the output of the algorithm is $\mathbf w$.  
\end{lemma}

\begin{proof} Algorithm~\ref{map:inverse} undoes Algorithm~\ref{map:forward},
one step at a time. \end{proof}

Thanks to the previous lemma, for any $(\mathbf w, \mathbf b)$ in 
$\mathbb I_m^k$, we can define $q(\mathbf w,\mathbf b)$ to be the result of
applying Algorithm~\ref{map:inverse} to $(\mathbf w,\mathbf b)$, and we
have that $q(p(\mathbf w))=\mathbf w$.  

We can restate the previous lemma in the following way:

\begin{corollary} The maps $p$ and $q$ are mutually inverse bijections
between $\mathbb W^k_m$ and $\mathbb I^k_m$.  \end{corollary}

\begin{proof} The fact that $q\circ p$ is the identity implies that 
$p$ is injective, and $p$ is surjective by definition.  Thus $p$ is a 
bijection, and $q$ is its inverse. \end{proof}

%By construction, we see that Algorithms~\ref{map:forward} and~\ref{map:inverse} are inverses.

%\begin{proposition}
%\label{prop:actsasinverse}
%Algorithm~\ref{map:inverse} acts as $(w^*)^{-1}$ on successful partitions.
%\end{proposition}

\vspace{20pt}

We now inductively define a tree of words on which Algorithm~\ref{map:inverse} succeeds.  
%The surprise is that this construction turns out to give all words in the image of Algorithm~\ref{map:forward}.

\begin{definition}\rm
\label{def:tree}
Define an infinite complete $m$-ary tree $\mathcal{T}_m^*$ by
\begin{enumerate}
	\item The $0$th rank consists of the empty word $\mathbf{\cdot}$, partitioned as $\underset{1}{\cdot} | \underset{2}{\cdot} | \cdots | \underset{0}{\cdot}$.
	\item The children of a partitioned word $(\mathbf{w},\mathbf{b})$ are the $m$ words obtained by prepending $-i \pmod m$ to $(\mathbf{w},\mathbf{b})_i$.
\end{enumerate}
\end{definition}

Figure~\ref{fig:tree} gives the first few rows of $\mathcal{T}_3^*$.

\begin{lemma} \label{lemma:all} $\mathcal T_m^*$ consists of all the words on which 
Algorithm~\ref{map:inverse} succeeds.\end{lemma}

\begin{proof} We first establish that Algorithm~\ref{map:inverse} succeeds
on every word in $\mathcal T_m^*$.  The proof is by induction on the rank in
the tree.  
Suppose that $(\mathbf w,\mathbf b)$ is in
$\mathcal T_m^*$, so it was obtained by prepending $-i$ (mod $m$) to
some word $(\mathbf {w'},\mathbf {b'})$ in the previous rank of 
$\mathcal T_m^*$.  It is straightforward to see that 
after the first iteration through the main loop of Algorithm~\ref{map:inverse},
the updated value of $(\mathbf w,\mathbf b)$ is $(\mathbf {w'},\mathbf {b'})$,
and the desired result follows by induction.  

We now prove the converse, that if Algorithm~\ref{map:inverse} succeeds
on $(\mathbf w,\mathbf b)$, then $(\mathbf w,\mathbf b)$ appears in
$\mathcal T^*_m$.  The proof is by induction on the length of 
$\mathbf w$.   Suppose that Algorithm~\ref{map:inverse} succeeds on
$(\mathbf w,\mathbf b)$, and suppose that $\mathbf w$ has positive length.
Since  
Algorithm~\ref{map:inverse} does not halt on the first step, then on the first
step it removes a single letter from $(\mathbf w,\mathbf b)$, obtaining
some successful word $(\mathbf {w'},\mathbf {b'})$.  By the induction hypothesis, this word appears
in $\mathcal T_m^*$, and the tree was defined in such a way that
the children of $(\mathbf {w'},\mathbf {b'})$ include $(\mathbf w,\mathbf b)$.
\end{proof}

%This definition is contrived so that Algorithm~\ref{map:inverse} succeeds on words in $\mathcal{T}_m^*$.
%\begin{enumerate}
%	\item If $\mathbf{\underline{w'}}$ is formed by prepending $-i$ to the $i$th block of $\mathbf{\underline{w}}$, then the block labeled by $i$ in $\mathbf{\underline{w}}$ becomes the block labeled by 0 in $\mathbf{\underline{w'}}$, and
%	\item The inverse map starts with the prepended $-i$ and returns us to performing $w^{-1}$ on $\mathbf{\underline{w}}$.
%\end{enumerate} 

The next step in our argument is the following proposition:

\begin{proposition}\label{prop:success}
Any $\mathbf w\in \mathbb W^k_m$ admits a successful partition
$(\mathbf w,\mathbf b)$.  
\end{proposition}

We defer the proof of this result to the
next section.

It is clear that the $k$-th row of $\mathcal T$ consists of $m^k$ partitioned
words.  According to Proposition~\ref{prop:success}, 
for each word $\mathbf w$ of length $k$, we are able to find a 
successful partition $(\mathbf w,\mathbf b)$.  This accounts for $m^k$ distinct partitioned words of length $k$ on which 
Algorithm~\ref{map:inverse} succeeds, and therefore, by Lemma~\ref{lemma:all},
this must be all of them.  It then follows that there is a 
unique successful partition of any word.  (We have also found a direct
proof of the uniqueness of the successful partition of $\mathbf w$, but
it was fairly involved, so we preferred not to present it.)
At this point, we therefore have the following:

\begin{proposition}[Dendrodistinctivity]
\label{prop:dendro}
Let $\mathcal{T}_m$ be the infinite complete $m$-ary tree obtained from $\mathcal{T}_m^*$ by replacing each partitioned word $(\mathbf{w},\mathbf{b})$ with its underlying word $\mathbf{w}$.  Then each word of length $k+1$ on $\mathbb{Z}/m\mathbb{Z}$ appears exactly once.
\end{proposition}

Now, for $\mathbf w \in \mathbb W_m^k$, we can define $g(\mathbf w)$
to be the unique successful partition of $\mathbf w$.  Since we know that
$p(\mathbf w)$ is a successful partition of $\mathbf w$, and there is only one,
it must be that $g(\mathbf w)=p(\mathbf w)$.  Therefore $g(\mathbf w)\in
\mathbb I^k_m$, and $q(g(f(p(\mathbf w))))=\mathbf w$.  

This establishes the following theorem:

\begin{theorem}\label{theorem:bijections}
The maps 
$(f\circ p)$ and $(q \circ g)$ are mutually inverse bijections from
$\mathbb W_m^k$ to $\mathbb W_m^k$.
\end{theorem}  

%The advantage of constructing this tree is that we automatically obtain the uniqueness of successful partitions if we are able to construct one for every word in $\mathbb{W}_m^k$.  This is because there are exactly $m^{k+1}$ words on the $(k+1)$st rank of $\mathcal{T}_m^*$, and because a successful partition shows us where a word lies in the tree.  It is also possible to give a direct argument of uniqueness.% that $w^*(\mathbb{W}_m^k) = (\mathcal{T}_m^k)^*$.

\begin{figure}[ht]

\begin{center}
\resizebox*{!}{2in}{$$\xymatrix @-1pc {
& & & & \cdot | \cdot | \cdot \ar@{-}[dd]\ar@{-}[ddlll]\ar@{-}[ddrrr] & & & & \\
& & & & & & & & & \\
& \cdot | \cdot | 0 \ar@{-}[dd]\ar@{-}[ddl]\ar@{-}[ddr]& & & \cdot | 1 | \cdot \ar@{-}[dd]\ar@{-}[ddl]\ar@{-}[ddr] & & & 2 | \cdot | \cdot \ar@{-}[dd]\ar@{-}[ddl]\ar@{-}[ddr] &\\
& & & & & & & & & \\
\cdot | \cdot | 00 & \cdot | 1 | 0 & 2 | \cdot | 0 & \cdot | 01 | \cdot & 1 | 1 | \cdot & \cdot | 1 | 2 & 02 | \cdot | \cdot & 2 | \cdot | 1 & 2 | 2 | \cdot \\
\vdots & \vdots & \vdots & \vdots & \vdots & \vdots & \vdots & \vdots & \vdots \\ 
}$$}

%\resizebox*{!}{2in}{$$\xymatrix @-1pc {
%& & & & \cdot \ar@{-}[dd]\ar@{-}[ddlll]\ar@{-}[ddrrr] & & & & \\
%& & & & & & & & & \\
%& 0 \ar@{-}[dd]\ar@{-}[ddl]\ar@{-}[ddr]& & &  1 \ar@{-}[dd]\ar@{-}[ddl]\ar@{-}[ddr] & & & 2 \ar@{-}[dd]\ar@{-}[ddl]\ar@{-}[ddr] &\\
%& & & & & & & & & \\
% 00 &  10 & 20 & 01 & 11 & 12 & 02 & 21 & 22 \\
%\vdots & \vdots & \vdots & \vdots & \vdots & \vdots & \vdots & \vdots & \vdots \\ 
%}$$}
\end{center}

\caption{The first few ranks of $\mathcal{T}_3^*$.}% and $\mathcal{T}_3$.}
\label{fig:tree}
\end{figure}
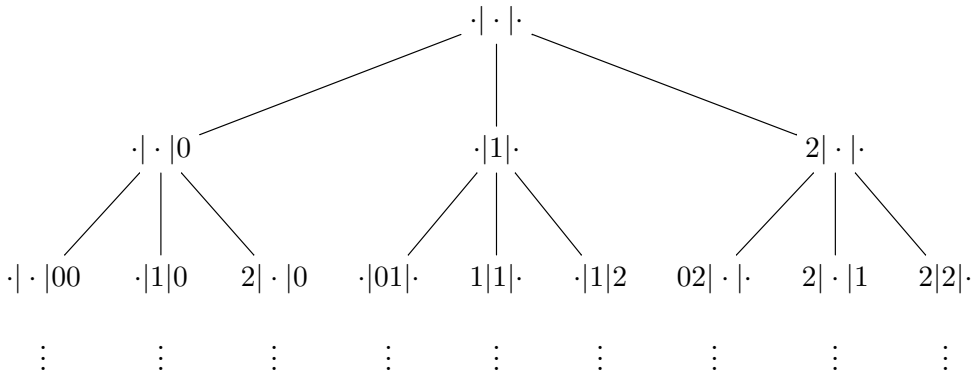

\vspace{20pt}
\section{Proof of Proposition~\ref{prop:success}}

In this section we provide the proof of Proposition~\ref{prop:success},
deferred from the previous section.

We begin our construction of a successful partition $(\mathbf{w},\mathbf{b^S})$ by encoding a partitioned word in a revealing way.   Given an index $i$, we find the position of the block to which the letter $w_i$ belongs by defining $\block(\mathbf{b},i)=t$, where $b_{t-1}+1 \leq i \leq b_t$ (and we take $b_0=0$).

\begin{definition} \rm
Define the $(k+1)\times m$ \emph{balancing matrix} $M_{(\mathbf{w},\mathbf{b})} = (m_{i,j})$ of a partitioned word $(\mathbf{w},\mathbf{b})$ by 
%m_{i,j} = 1 if b(\mathbf{w},i)\leq j \leq b(\mathbf{w},i)+w_i-1, 0 otherwise.
$$m_{i,j} = \begin{cases}
1 & \text{if }\block(\mathbf{b},i)\leq j \leq \block(\mathbf{b},i)+w_i-1, \\
0 & \text{otherwise.}
\end{cases}$$
\end{definition}

\begin{definition} \rm
We say that column $j$ of $M_{(\mathbf{w},\mathbf{b})}$ is \emph{equitably filled} if it has $\lfloor \frac{1}{m} \sum_{i=1}^{k+1} w_i \rfloor+1$ ones and $j \in \{m-\sigma,m-\sigma+1,\ldots,m-1\}$, or if it has $\lfloor \frac{1}{m} \sum_{i=1}^{k+1} w_i \rfloor$ ones and $j \not \in \{m-\sigma,m-\sigma+1,\ldots,m-1\}$.  We say that $M_{(\mathbf{w},\mathbf{b})}$ is \emph{equitably distributed} if each of its columns is equitably filled.
%a balancing matrix $M_{\mathbf{b(w)}}$ is \emph{evenly distributed} if there are $\lfloor \frac{\sum_{i=1}^{k+1} w_i}{m} \rfloor+1$ ones in columns $m-\sigma,m-\sigma+1,\ldots,m-1$, and $\lfloor \frac{\sum_{i=1}^{k+1} w_i}{m} \rfloor$ ones in each other column.  We say that a column is evenly distributed if 
If $M_{(\mathbf{w},\mathbf{b})}$ is equitably distributed, we call $(\mathbf{w},\mathbf{b})$ (or just $\mathbf{b}$ when $\mathbf{w}$ is understood) an \emph{equitable partition} for $\mathbf{w}$.
\end{definition}

Note that when $(\mathbf{w},\mathbf{b})$ is a successful partition, then it is also an equitable partition.  The converse is false---see Figure~\ref{fig:matsuc} for an example.

\begin{figure}[ht]
\begin{center}
\begin{tabular}{cc}
$\left( \begin{array}{c|cccc}
3 &1&1&1&\\ \hline
2 &&1&1&\\ \hline
1 &&&1&\\ \hline
0 &&&&\\ 
3 &1&1&&1\\
0 &&&&\\
2 &1&&&1\\
\end{array} \right)$ &
$\left( \begin{array}{c|cccc}
3 &1&1&1&\\
2 &1&1&&\\
1 &1&&&\\
0 &&&&\\ \hline
3 &&1&1&1\\
0 &&&&\\ \hline
2 &&&1&1\\ \hline
\end{array} \right)$
\end{tabular}
\end{center}
\caption{$M_{(\mathbf{w},\mathbf{b})}$ for the successful partition $3|2|1|0302$ and for the (rightmost but not successful) equitable partition $3210|30|2|\cdot$.}
\label{fig:matsuc}
\end{figure}

%example: even but not successful
\iffalse$$\begin{array}{ccccc}

2&X&X&&\\
0&\\
0&\\
3&X&X&X\\
\hline1&&X\\
\hline0&&&\\
3&X&&X&X\\
1&&&X\\
\hline1&&&&X
\end{array}$$\fi

We now construct a specific equitable partition, which is our first approximation to the successful partition $(\mathbf{w},\mathbf{b^S})$.

\begin{definition} \rm
Let a \emph{rightmost equitable partition} be an equitable partition $(\mathbf{w}, \mathbf{b})$ such that for any other equitable partition $(\mathbf{w},\mathbf{b'})$, we have $b_t \geq b'_t$ for $0\leq t \leq m-1$. 
\end{definition}

We prove existence and uniqueness of the rightmost equitable partition.

\begin{lemma}
\label{evenprop}
For any word $\mathbf{w}$, there is a unique rightmost equitable partition $(\mathbf{w},\mathbf{b^R})$.
\end{lemma}

\begin{proof}
We claim that Algorithm~\ref{map:rightmost} constructs the unique rightmost equitable partition $(\mathbf{w},\mathbf{b^R})$.

\begin{algorithm}[ht]
\KwIn{A word $\mathbf{w}$.}
\KwOut{The rightmost equitable partition $(\mathbf{w},\mathbf{b^R})$.}
Set $(\mathbf{w},\mathbf{b^R})_{\sigma+1}:= \mathbf{w}$ and $(\mathbf{w},\mathbf{b^R})_{\sigma+2}:=(\mathbf{w},\mathbf{b^R})_{\sigma+3}:=\cdots:=(\mathbf{w},\mathbf{b^R})_{\sigma}:=\emptyset $\;
\While{$M_{(\mathbf{w},\mathbf{b^R})}$ is not equitably distributed}{
		Let column $t$ be the first non-equitably filled column of $M_{(\mathbf{w},\mathbf{b^R})}$\;
		Let $tmp$ be the rightmost element in $(\mathbf{w},\mathbf{b^R})_{\sigma+t}$\;
		Delete $tmp$ from $(\mathbf{w},\mathbf{b^R})_{\sigma+t}$\;
		Prepend $tmp$ to $(\mathbf{w},\mathbf{b^R})_{\sigma+t+1}$\;
}
Return $(\mathbf{w},\mathbf{b^R})$\;
\caption{The rightmost equitable partition.}
\label{map:rightmost}
\end{algorithm}

By construction---assuming Algorithm~\ref{map:rightmost} is well-defined---it returns the unique rightmost equitable partition, because at each step of 
Algorithm~\ref{map:rightmost}, if a letter is moved from the $i$-th part
to the $i+1$-st part, then it necessarily occurs to the right of the $i$-th
part in any equitable partition.

%   so that it terminates in the unique rightmost even partitioning as long as it does not get stuck.  
%If there are too many X's in a column, we must eventually push some of the
%letters in that part into a subsequent column:
%there is no chance that making other moves will
%fix the problem for us --- in fact, other moves can only make the situation
%worse.  
%Thus, if our algorithm does not

There are two ways that Algorithm~\ref{map:rightmost} might fail.  The first is if every column but the rightmost of $M_{(\mathbf{w},\mathbf{b^R})}$ is equitably filled, since then the algorithm then tries to push the rightmost entry of block $(\mathbf{w},\mathbf{b^R})_\sigma$ to the leftmost entry of $(\mathbf{w},\mathbf{b^R})_{\sigma+1}$, which would destroy the property that $\mathbf{b^R}$ is a partition of $\mathbf{w}$.  But if all other columns are equitably filled, then it follows from the definition that column $m$ is also equitably filled.  Therefore, this case does not occur.

%We should therefore ask what would cause the algorithm to get stuck.  We 
%would certainly be in trouble if we had too many X's in the 
%last column, because we can't move numbers further than the last part. 
%However, this isn't a risk, because the algorithm preserves the starting
%property that the total number of X's in the first $i$ columns is
%at least as many as we want,
%for all $i$.  
%If we have enough X's in the first $m$ columns, we cannot also have too
%many X's in the last column.  

The second way for Algorithm~\ref{map:rightmost} to fail to be well-defined is if it tries to push the rightmost entry of an empty block $(\mathbf{w},\mathbf{b^R})_{\sigma+t_0}$ to the block $(\mathbf{w},\mathbf{b^R})_{\sigma+t_0+1}$.  By assumption, we know that columns $t$ for $1\leq t < t_0$ of $M_{(\mathbf{w},\mathbf{b^R})}$ are all equitably filled.  Since $(\mathbf{w},\mathbf{b^R})_{\sigma+t_0}$ is empty, we know that the number of ones in column $t_0$ is at most the number of ones in column $t_0-1$.  By the definition of equitably filled, column $t$ must have the same number of ones as column $t-1$ for $t\neq m-\sigma$ and $t\neq m$.  But for $t = m-\sigma$, column $t$ must have exactly one more one than column $t-1$, and for $t=m$ we find ourselves back in the previous case.  Then column $t_0$ must have been equitably filled so that this case also does not occur.

%We might also worry that we could find ourselves wanting to remove a letter
%from a part in which there actually are no letters.  That is to say, the
%issue would be that the previous parts were already contributing too many
%X's into the $c$-th column, so that we would want to remove a letter from
%the $c$-th part, which was already empty.  Note, though, that the 
%number of X's contributed to a column from other parts, is no
%more than the number of X's in the column to its left.  Thus this problem
%only potentially arises if, in the even distribution, we want to have fewer
%X's in column $c$ than in column $c-1$.  But 
%the even distribution is set up in such a way that this only happens for
%$c=m$, and we know that we will never have too many X's in the last column.

%0,1,2,3...m-1
%1,2,3,...

%m-1,0,1,...,m-2
%1

\end{proof}

An example of the rightmost partition occurs as the rightmost example in Figure~\ref{fig:matsuc}.

%If we run this algorithm on our example $200310311$, the result is 
%$200310|3|1|1$.  Note that this is not a position of the dividers which arises 
%in the tree, so we still have more work to do.
  
%\begin{proposition}
%\label{exist}
%For any word $\w$, there is at least one way to insert dividers which produces a word as it appears in the tree.
%\end{proposition}

We now prove Proposition~\ref{prop:success}.% that every word of length $k+1$ on $\mathbb{Z}/m\mathbb{Z}$ has a successful partitioning, which implies that every such word has a unique successful partitioning.% from which we conclude that every such word appears exactly once in $T_m$.

\begin{proof}

We claim that Algorithm~\ref{map:success} constructs the unique successful partition $(\mathbf{w},\mathbf{b^S})$.

\begin{algorithm}[ht]
\KwIn{A word $\mathbf{w}$.}
\KwOut{The successful partition $(\mathbf{w},\mathbf{b^S})$.}
Let $(\mathbf{w},\mathbf{b^S})$ be the output $(\mathbf{w},\mathbf{b^R})$ from Algorithm~\ref{map:rightmost} applied to $\mathbf{w}$\;
\While{$(\mathbf{w},\mathbf{b^S})$ is not a successful partition}{
	Let $(\mathbf{w'},\mathbf{b'})$ be the second part of the output from Algorithm~\ref{map:inverse} applied to $(\mathbf{w},\mathbf{b^S})$\;
	\For{$t=1$ \KwTo $m-1$}{
		Delete the rightmost $|(\mathbf{w'},\mathbf{b'})_{\sigma+t}|$ entries from $(\mathbf{w},\mathbf{b^S})_{\sigma+t}$\;
		Prepend $(\mathbf{w'},\mathbf{b'})_{\sigma+t}$ to $(\mathbf{w},\mathbf{b^S})_{\sigma+t+1}$\;
	}
}	
%Set $\mathbf{b}_{\sigma+1} = \mathbf{w}$ and $\mathbf{b}_{\sigma+2}=\mathbf{b}_{\sigma+3}=\cdots=\mathbf{b}_{\sigma}=\{\}$\;
%\For{$t=1$ \KwTo $m$}{
%	\If{column $t$ of $M_\mathbf{b(w)}$ is not evenly distributed}{
%		Let $tmp$ be the rightmost element in $\mathbf{b}_{\sigma+t}$\;
%		Delete $tmp$ from $\mathbf{b}_{\sigma+t}$\;
%		Prepend $tmp$ to $\mathbf{b}_{\sigma+t+1}$\;
%		$t=0$\;
%	}
%}
Return $(\mathbf{w},\mathbf{b^S})$\;
\caption{The successful partition $f^{-1}(\mathbf{w})$.}
\label{map:success}
\end{algorithm}

By construction---assuming Algorithm~\ref{map:success} is well-defined---it returns a successful partition for any word $\mathbf{w}$ of length $k+1$.  
%But since rank $(k+1)$ of the tree $\mathcal{T}^*_m$ has $m^{k+1}$ words, and every word of length $(k+1)$ must therefore occur on that rank, it must be that this was the unique successful partition for the word $\mathbf{w}$.

%Begin with the rightmost position of the dividers which 
%produces an even distribution of X's.  As we saw in the example just
%above, this may not be the correct answer.  Attempt to remove letters 
%from this word.  As we remove letters, we are also going to remove the
%X's which they contribute.  
%It's clear which part to start in: either the one 
%corresponding to the leftmost
%column with the larger number of X's, or the last one if they are all equal.
%If you succeed, we are done.

To understand why this algorithm is well-defined, we will consider what the possibilities are for the second output of Algorithm~\ref{map:inverse}.  Algorithm~\ref{map:inverse} terminates when it tries to remove a letter from an empty block of $(\mathbf{w'},\mathbf{b'})$, which corresponds to a column in $M_{(\mathbf{w'},\mathbf{b'})}$ whose ones all came from other parts.  Furthermore, since Algorithm~\ref{map:inverse} removes entries while preserving the property that $M_{(\mathbf{w'},\mathbf{b'})}$ is equitably distributed, it must remove an entry corresponding to the leftmost column with the most number of ones or an entry corresponding to the rightmost column (if all columns have the same number of ones).  Putting these two pieces together, we see that Algorithm~\ref{map:inverse} terminates because it was trying to remove a letter from the rightmost block.  Shifting the letters remaining in $(\mathbf{w'},\mathbf{b'})$ while preserving those in $(\mathbf{w},\mathbf{b})$ that are not in $(\mathbf{w'},\mathbf{b'})$, we obtain a new equitable partition for $\mathbf{w}$.

Each time we repeat the process, the letters are moved further to the right, so
we cannot get back to an equitable partition which we had obtained 
previously.  Since there are only a finite number of equitable partitions,
the algorithm eventually terminates, finding a successful partition
$(\mathbf w,\mathbf b)$.  \end{proof}

%Since $M_{\mathbf{b'(w')}}$ is evenly distributed, 

%Suppose you get stuck.  That means you want
%to remove a letter from a part which is empty, which means that it 
%corresponds to a column all of whose X's come from a previous part (at this
%stage in the algorithm).  
%If you get stuck, 
%it must be in a state where there are an equal number of X's in each
%column, since otherwise, the column you are trying to remove from has
%more X's than its predecessor, so it has at least one letter in the
%corresponding part.  You are therefore stuck at a stage where you are trying
%to remove letters from the last column, so the last part is empty.  
%Note that, at the point where you got stuck, the remaining letters in each
%part are a final subsequence of the letters that were originally in that part.  
%Therefore, it is possible to move the dividers so that each of these 
%letters is in the next part to the right.  Since there were an equal number
%of X's in each column when we got stuck, this will still be true after the
%move, so the result will still be evenly distributed.  

%After applying this procedure $t$ times, the leftmost $t$ blocks of the output from Algorithm~\ref{map:inverse} will be empty.  
%It therefore takes at most $m$ iterations for this algorithm to return a successful partition $(\mathbf{w}, 
%\mathbf b)$.
%\end{proof}

Figure~\ref{fig:partition} gives an example of Algorithm~\ref{map:success}.

\begin{figure}[ht]

%We compute $\mathbf{b^S}(3210302)$ using Algorithm~\ref{map:success}.  We begin with the rightmost partitioning $3210|30|2|\cdot$.

\begin{tabular}{|c|c|}
\hline
The current partition $(\mathbf{w},\mathbf{b^S})$ & The output $(\mathbf{w'},\mathbf{b'})$ from Algorithm~\ref{map:inverse}\\ \hline
$(\mathbf{w},\mathbf{b^R}) = 3210|30|2|\cdot$ & $210|30|2|\cdot$ \\ \hline
$3|\mathbf{210}|\mathbf{30}|\mathbf{2}$ & $\cdot|10|30|\cdot$ \\ \hline
$3|2|\mathbf{10}|\mathbf{30}2$ & $\cdot|\cdot|0|\cdot$ \\ \hline
$3|2|1|\mathbf{0}302$ & $\cdot|\cdot|\cdot|\cdot$ \\ \hline
\end{tabular}

%{{{3},{2},{1},{0,3,0,2}},{{{2,1,0},{3,0},{2},{}},{{},{1,0},{3,0},{}},{{},{},{0},{}}}}

\caption{The construction of the successful partition $(3210302,\mathbf{b^S})$ using Algorithm~\ref{map:success}.}
\label{fig:partition}
\end{figure}

\vspace{30pt}
%------------------------------------------------------------------------------------------------------------------------
\section{The Conclusion of the Proof of Theorem~\ref{thm:forward}}
\label{sec:proof}
%------------------------------------------------------------------------------------------------------------------------

We conclude the proof of Theorem~\ref{thm:forward} by specializing the result from the previous section.

\begin{proposition}
$w$ is the same map as $(f \circ p)$ restricted to words in $\mathcal{X}_m^k$ with an $(m-1)$ appended.
\end{proposition}

\begin{proof}
Fix a word $\mathbf{x} = x_1 x_2 \ldots x_k (m-1)$ consisting of a word in
$\mathcal{X}_m^k$ with an $(m-1)$ appended.  To prove that
$w(\mathbf{x})=f(p(\mathbf{x}))$, we show that each letter of
$w(\mathbf{x})$ occurs in the same position as each letter of
$f(p(\mathbf{x}))$ by showing that they have the same number of letters to
their left.

In order to analyze the map $w$, note that exactly one of $x_i, x_{(k+1)+i}=x_i-1,
x_{2(k+1)+i}=x_i-2, \ldots, x_{(m-1)(k+1)+i}=x_i-(m-1)$ in
$\overline{\mathbf{x}(m-1)}$ occurs as the first letter in some word in 
the orbit of
$\mathbf{x}$.  This happens when $x_{j(k+1)+i-1}=0$, which means that
$j=x_{i-1}$ and so the $i$th position of $\mathbf{x}$ gives rise to
$x_i-x_{i-1}$
in $w(\mathbf{x})$.  
Given an extended word $\overline{x_1 x_2 \ldots x_k
(m-1)}$, a translate of $x_j< x_{i-1}$ will be zero to the left of when a
translate of $x_{i-1}$ is zero.  The same is true for $x_j = x_{i-1}$ with
$j<i-1$.  Then the letter $x_i-x_{i-1}$ in $w(\mathbf{x})$ occurs to the
right of all $x_j - x_{j-1}$ for which $x_j< x_{i-1}$ and all $x_j -
x_{j-1}$ for which $x_{j-1} = x_{i-1}$ and $j<i-1$.

On the other hand, since $x_{k+1}=m-1$, the blocks in
Algorithm~\ref{map:forward} are labeled from left to right by
$0,1,\ldots,m-1$.  All $x_j - x_{j-1}$ for which $x_j< x_{i-1}$ are placed
in blocks to the left of $b_{x_{i-1}}$ and all $x_j - x_{j-1}$ for which
$x_{j-1} = x_{i-1}$ and $j<i-1$ are added earlier to the block
$b_{x_{i-1}}$.  Replacing a partitioned word with its underlying word leaves
these letters to the left of $x_i-x_{i-1}$.

Each letter $x_i-x_{i-1}$ therefore occurs in $w(\mathbf{x})$ in the same
position as $(f \circ p)(\mathbf x)$, so that the two maps are identical.
%
%
%
%
% Fix a word $\mathbf{x} = x_1 x_2 \ldots x_k (m-1)$ consisting of a word 
% in $\mathcal{X}_m^k$ with an $(m-1)$ appended.
% %Exactly one of $x_i, x_{(k+1)+i}=x_i-1, x_{2(k+1)+i}=x_i-2, \ldots, x_{(m-1)(k+1)+i}=x_i-(m-1)$ in $\overbar{\mathbf{x}(m-1)}$ occurs as the first letter in the orbit of $\mathbf{x}$.  This happens when $x_{j(k+1)+i-1}=0$, which means that $j=x_{i-1}$ and so the $i$th position of $\mathbf{x}$ becomes $x_i-x_{i-1}$ in $\mathbf{w}$.
% %Given an extended word $\overbar{x_1 x_2 \ldots x_k (m-1)}$, a translate of $x_j< x_{i-1}$ will be zero to the left of when a translate of $x_{i-1}$ is zero.  The same is true for $x_j = x_{i-1}$ with $j<i-1$.
% Then $x_i-x_{i-1}$ occurs as the first letter of an element in the orbit of $\mathbf{x}$ under $\phi$ after all $x_j - x_{j-1}$ for which $x_j< x_{i-1}$ and all $x_j - x_{j-1}$ for which $x_{j-1} = x_{i-1}$ and $j<i-1$.

% On the other hand, when $x_{k+1}=m-1$, the blocks for Algorithm~\ref{map:forward} are labeled from left to right by $0,1,\ldots,m-1$.  All $x_j - x_{j-1}$ for which $x_j< x_{i-1}$ are placed in blocks to the left of $b_{x_{i-1}}$ and all $x_j - x_{j-1}$ for which $x_{j-1} = x_{i-1}$ and $j<i-1$ are added earlier to the block $b_{x_{i-1}}$.

% Each letter $x_i-x_{i-1}$ therefore occurs in $w(\mathbf{w})$ in the same position as $(f \circ p)$, so that the two maps are identical.
\end{proof}

\begin{proposition} The map $w$ is invertible; its inverse is
$(q\circ g)|_{\mathcal{W}_m^k}$.
\end{proposition}

\begin{proof}
This is a specialization of Theorem~\ref{theorem:bijections}, combined with the identification of the image of $w$ as $\mathcal{W}_m^k$.
\end{proof}

These two propositions conclude the proof of Theorem~\ref{thm:forward}.

\section{An Application to Parking Functions}
\label{sec:parkingfunctions}
%------------------------------------------------------------------------------------------------------------------------

In this section we use our results to define a new labeling of regions of the $m$-Shi arrangement with $m$-parking functions.  This partially answers a question of D.\ Stanton, who asked for a reason that it might be natural to biject the alcoves $\mathcal{Z}_m^k$ with the set of words $\mathcal{W}_m^k$.

\begin{definition} \rm
An \emph{$m$-parking function} of length $k$ is a word $a_1 a_2 \ldots a_k$ with $a_i \in \mathbb{N}$ such that $|\{ a_j | a_j > m i \}| \leq k-i$ for $0 \leq i \leq k$.  In what follows, we will normally suppress the mention of $k$.  
% satisfying the following condition: if $b1 \leq b2 \leq \cdots \leq bn$ is the monotonic rearrangement of the terms of, then b_i \leq k(i-1)
\end{definition}

%\begin{theorem}[ref?]
%There are $(km+1)^{m-1}$ $k$-parking functions.
%\end{theorem}

\begin{definition} \rm
For $m \geq 1$, the $m$-Shi arrangement $\mathcal{S}_{km+1}^{k-1}$ is the collection of hyperplanes $x_i-x_j = s$ for $1 \leq i < j \leq k$ and $-m+1\leq s \leq m$.
\end{definition}

\begin{theorem}[I.\ Pak and R.\ Stanley~\cite{stanley1996hyperplane, stanley1998hyperplane}]
	There is a bijection between regions of the $m$-Shi arrangement and $m$-parking functions.
\end{theorem}

The Pak-Stanley map $\lambda$ that labels $m$-Shi regions with $m$-parking functions is particularly easy to state \cite{stanley1998hyperplane}.  Let $e_i$ be the word with a one in the $i$th position and zeroes elsewhere.  Label the fundamental alcove by the $m$-parking function $00\ldots0$.  When a region $R$ has been labeled, and $R'$ is an unlabeled region that is separated from $R$ by a unique hyperplane $x_i-x_j = s$ with $i < j$, the new region is labeled by

\begin{align*}
\lambda(R') & = \lambda(R) + e_j  \text{ if } s \leq 0,\\
\lambda(R') &= \lambda(R) + e_i \text{ if } s > 0.
\end{align*}

An example of this labeling is given in Figure~\ref{fig:shiwithpf}.
The inverse map did not appear until two years 
later~\cite{stanley1998hyperplane}, and is more involved.  

\begin{figure}[ht]
\begin{center}
\includegraphics[height=3.3in]{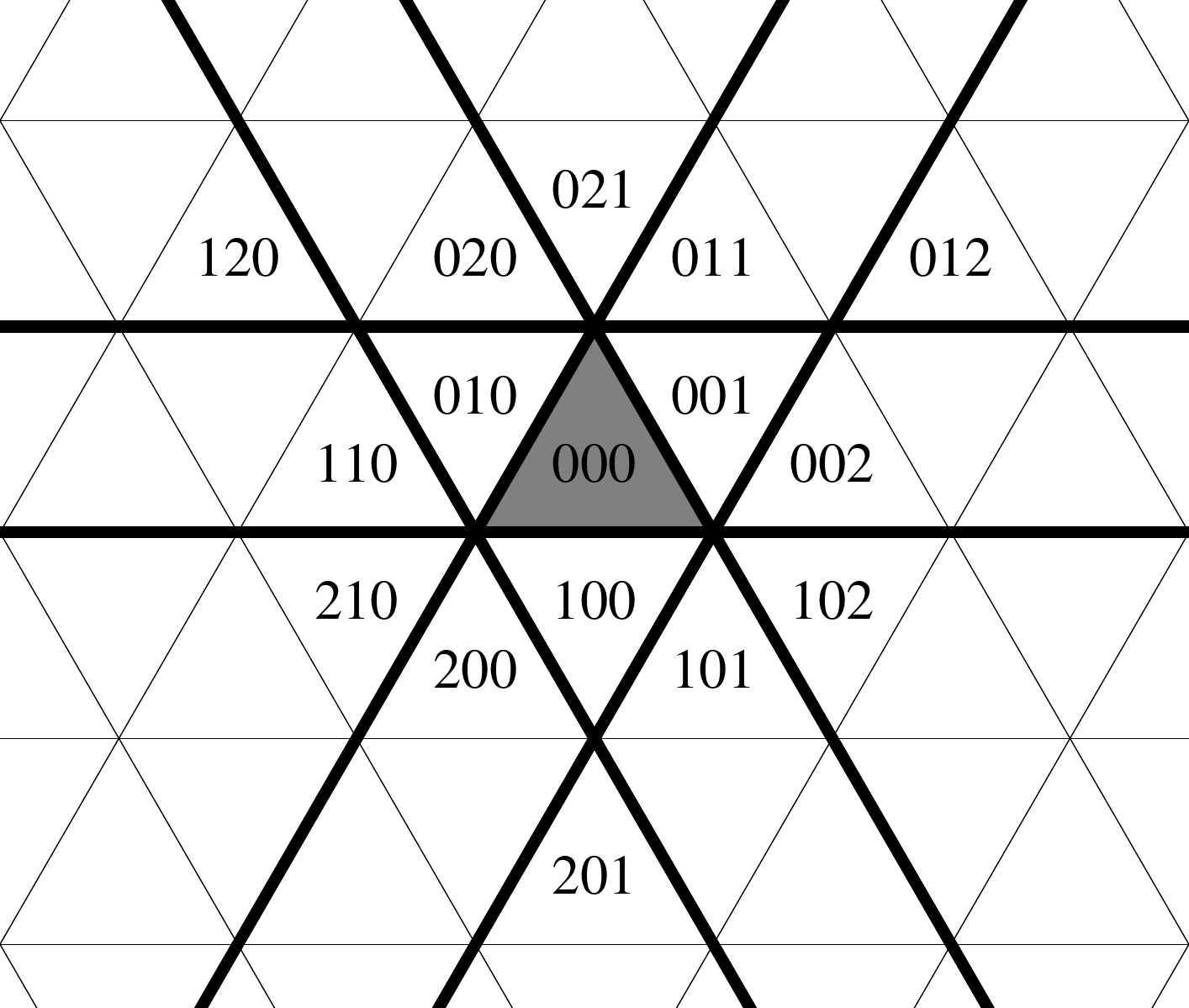}
\end{center}
\caption{The 16 minimal alcoves of the regions of $\mathcal{S}_4^2$ labeled with their corresponding parking functions under the Pak-Stanley labeling.}
\label{fig:shiwithpf}
\end{figure}

%It is interesting to note that while the forward map (for $1$-parking functions) that I.\ Pak and R.\ Stanley give in~\cite{stanley1996hyperplane} is easy to describe, the inverse map (even for $1$-parking functions) did not appear until two years later~\cite{stanley1998hyperplane}.  As a specialization of our results will turn out to be equivalent to theirs, this is a strong indication that our inverse map in Section~\ref{sec:dendro} cannot be simplified much.

%``Not so evident is the following result obtained in collaboration with Igor Pak'' (stanley1996hyperplane)
%``Some similarities between the Shi arrangement and the extended Shi arrangements were pointed out by A. Postnikov after which it was straightforward to extend Pak's method of labeling. (However, there remained the problem of actually proving that Pak's labeling and its extension to S_n^k had the desired properties.)'' (stanley1998hyperplane)

We now describe the labeling of regions of the $m$-Shi arrangement by
$m$-parking functions which follows from our perspective.  

For $1\leq i\leq k$, let $s_i$ be the \emph{reflection} in the hyperplane $H_{\alpha_i,0}$ and let $s_0$ be the reflection in $H_{-\alpha_0,1}$.  
As we have already remarked, the reflections $s_0,\dots,s_k$ generate the 
\emph{affine symmetric group}.
%The group $\{s_i\}_{i=0}^k$ is the \emph{affine symmetric group}, which has relations
%\begin{align*}
%s_i^2 & = 1 \text{ for } i \in [k],\\
%s_i s_j &= s_j s_i \text{ if } i-j \neq \pm 1,\\
%s_i s_{i+1} s_i &= s_{i+1} s_i s_{i+1} \text{ for } i \in [k].
%\end{align*}

The affine symmetric group acts simply transitively on the set of alcoves
in the $A_k$ affine hyperplane arrangement.  (For definiteness, we let 
the affine symmetric group act on the left.)  We can therefore label
each alcoves by the unique affine permutation which takes the fundamental
alcove to that alcove.  
%It is clear that we may therefore uniquely label an alcove of the type $A_k$ affine hyperplane arrangement with an affine permutation.  

Since the affine permutations form a group, we may find the inverse permutation and the corresponding inverse alcove.  C.\ Athanasiadis and E.\ Sommers proved that there is a unique minimal alcove from the type $A_k$ affine hyperplane arrangement in each region of the $m$-Shi arrangement~\cite{athanasiadis2005refinement,sommers2003b}.  Extending work of J.-Y. Shi in~\cite{shi1987sign}, E.\ Sommers showed that the collection of the inverted minimal alcoves from the $m$-Shi arrangement forms a $(km+1)$-fold dilation of the fundamental alcove~\cite{sommers2003b}.   This dilation turns out to be a translation of $\mathcal{Z}_{km+1}^{k-1}$---the fundamental alcove sits in the middle, rather than at the edge of this simplex.  The authors are grateful to D.\ Amstrong for compiling this story in~\cite{armstrong2010hyperplane} and to V.\ Reiner for pointing them in this direction.

An example of the Pak-Stanley labeling on the inverse alcoves is given in Figure~\ref{fig:invshiwithpf}.

\begin{figure}[ht]
\begin{center}
\includegraphics[height=3.3in]{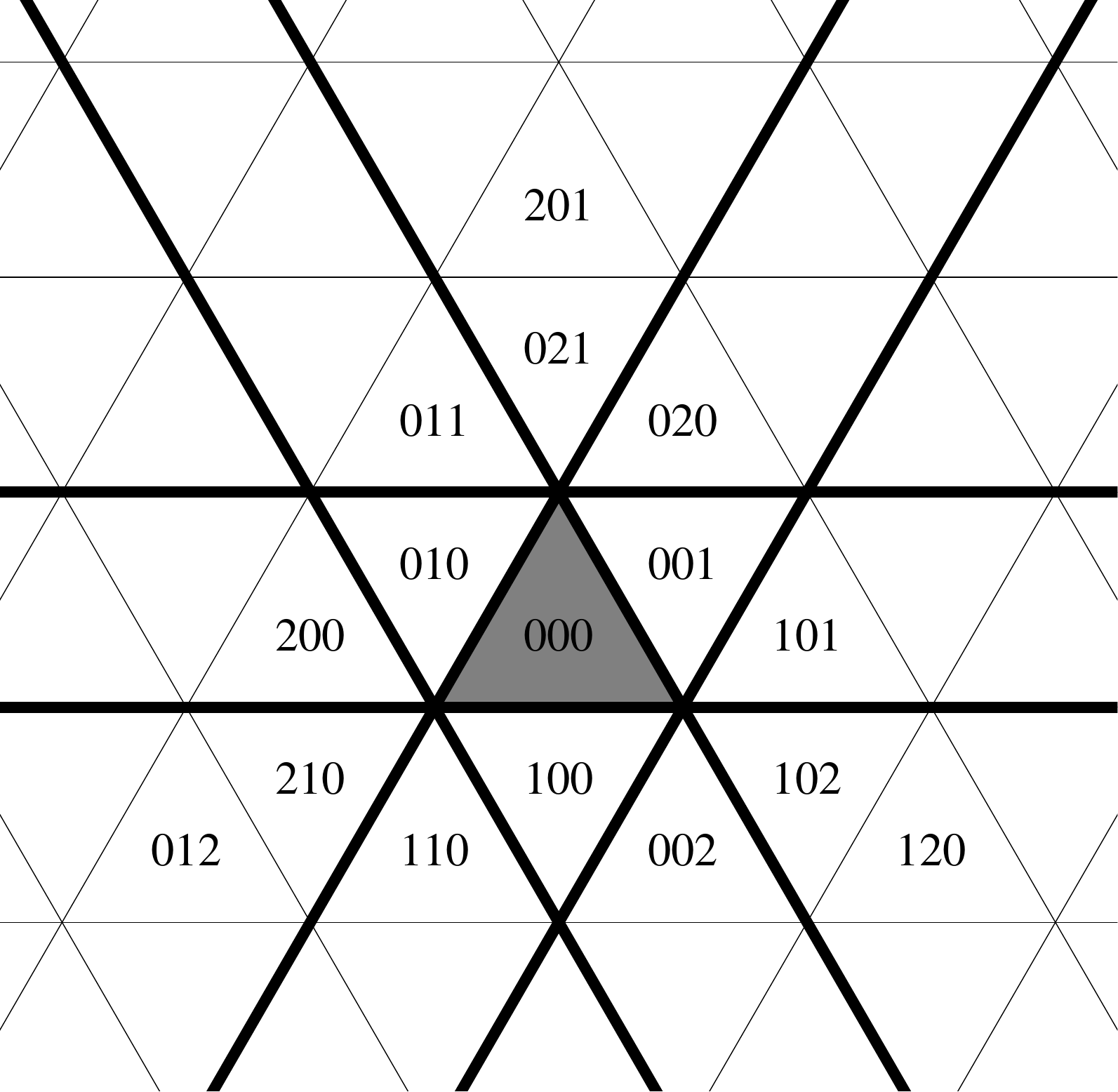}
\end{center}
\caption{The inverses of the 16 minimal alcoves of the regions of $\mathcal{S}_4^2$ labeled with their corresponding parking functions under the Pak-Stanley
labeling.}
\label{fig:invshiwithpf}
\end{figure}

The standard proof of the enumeration formula for parking functions  notes that every coset of the subgroup of $\mathbb{Z}_{km+1}^k$ generated by $(1,1,...,1)$ contains exactly one parking function~\cite{foata1974mappings}.  But notice that every coset also evidently contains one word that sums to $km \mod km+1$.  Starting with our labeling of the scaled simplex with words that sum to $km \mod km+1$ (Figure~\ref{fig:ex1}), we may therefore select the parking function in the same coset (the left part of Figure~\ref{fig:oldinvshiwithpf}), translate the simplex, and then find the inverse alcoves to give a labeling of the $m$-Shi arrangement by $m$-parking functions (the right part of Figure~\ref{fig:oldinvshiwithpf}).  Note that this is a different labeling from those given in~\cite{stanley1996hyperplane} and~\cite{athanasiadis1999simple} and is ---to the best of our knowledge--- new.

An immediate corollary is a CSP for $m$-parking functions under rotation and the regions of the $m$-Shi arrangement under its cyclic symmetry.

\begin{figure}[ht]
\begin{center}
$\begin{array}{cc}
\includegraphics[height=2.5in]{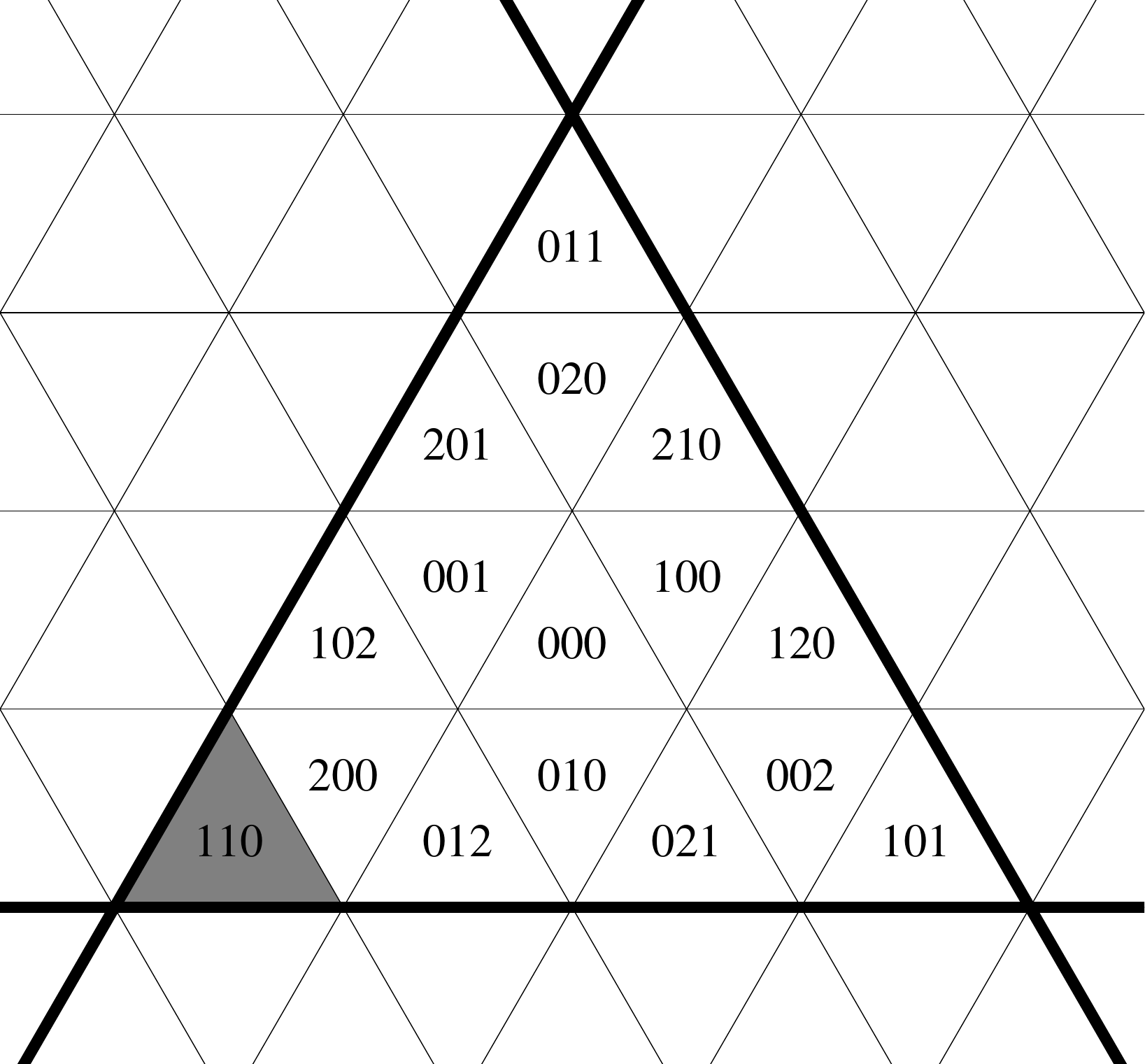} & \includegraphics[height=2.5in]{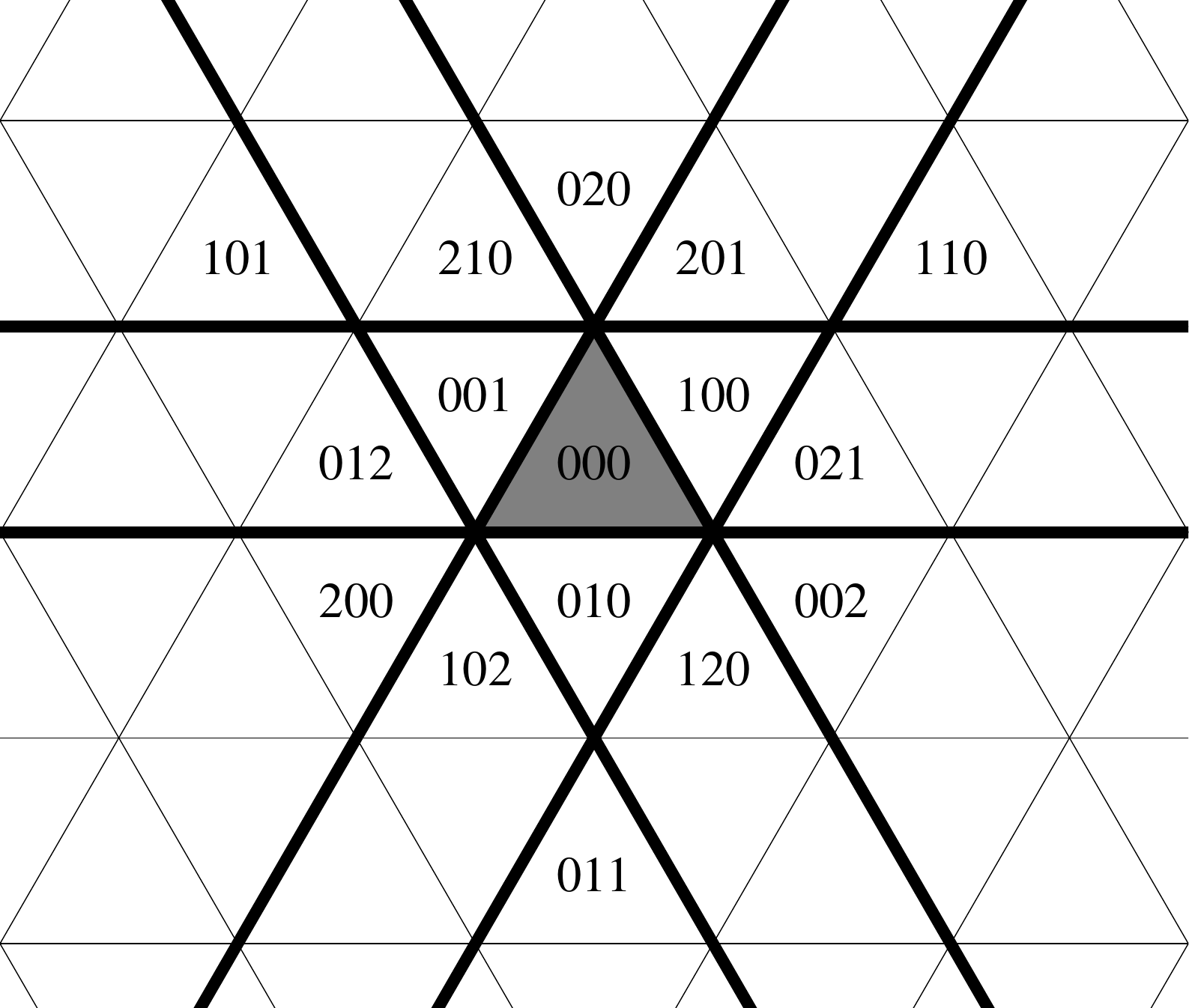} \\
\end{array}$
\end{center}
\caption{On the left are the 16 alcoves of $\mathcal{Z}_4^2$ labeled with the parking functions coming from words that sum to $3 \mod 4$.  On the right are the 16 inverses of these alcoves after translation---which are the minimal alcoves of the regions of $\mathcal{S}_4^2$---labeled in the same way.}
\label{fig:oldinvshiwithpf}
\end{figure}

%------------------------------------------------------------------------------------------------------------------------
\section{Acknowledgments}
\label{sec:ack}
%------------------------------------------------------------------------------------------------------------------------

The authors thank L.\ Serrano and the organizers of the Algebraic Combinatorics meets Combinatorial Algebra conference at UQAM for inviting and introducing them to each other.  They are especially indebted to C.\ Berg and M.\ Zabrocki for showing them how this problem arises from geometry, M.\ Zabrocki for helpful discussions and his interest in dendrodistinctivity, C.\ Berg his suggestions on the proof of Theorem~\ref{thm:coretocombo} and careful reading, and M.\ Visontai for innumerable suggestions regarding proof and presentation.  N.\ Williams is grateful to V.\ Reiner and D.\ Stanton for their guidance and patience, and thanks them for giving him this problem.

\bibliographystyle{amsplain}
\bibliography{CSOTSS}

\end{document}